\documentclass[12pt]{amsart}
\usepackage{amsfonts}
\usepackage{amssymb,latexsym}
\usepackage{amsmath}
\usepackage{amsthm}
\usepackage{amssymb}
\usepackage{enumerate}
\newtheorem{theorem}{Theorem}[section]
\newtheorem{lemma}[theorem]{Lemma}
\newtheorem{proposition}[theorem]{Proposition}

\theoremstyle{definition}

\theoremstyle{remark}
\newtheorem{remark}[theorem]{Remark}
\numberwithin{equation}{section}

\begin{document}
\title
{The Neumann problem of Hessian quotient equations}
\author{Chuanqiang Chen$^1$, Dekai Zhang$^2$}
\address{Chuanqiang Chen, School of Mathematics and Statistics, Ningbo University, Ningbo, 315211, Zhejiang Province, P.R. China}
\address{Dekai Zhang, Shanghai Center for Mathematical Sciences, Jiangwan Campus, Fudan University, No. 2005 Songhu Road, Shanghai, P.R. CHINA}
\thanks{ E-mail: $^1$ chuanqiangchen@zjut.edu.cn, $^2$  dkzhang@fudan.edu.cn}
\thanks{$^1$ Research of the first author was supported by NSFC NO. 11771396.}

\maketitle

\begin{abstract}
In this paper, we obtain some important inequalities of Hessian quotient operators, and global $C^2$ estimates of the Neumann problem of Hessian quotient equations. By the method of continuity, we establish the existence theorem of $k$-admissible solutions of the Neumann problem of Hessian quotient equations.
\end{abstract}


\section{Introduction}

In this paper, we consider the $k$-admissible solution of the Neumann problem of Hessian quotient equations
\begin{align} \label{1.1}
\frac{\sigma _k (D^2 u)}{\sigma _l (D^2 u)} = f(x),  \quad \text{in} \quad \Omega \subset \mathbb{R}^n,
\end{align}
where $0 \leq l < k \leq n$, and for any $k = 1, \cdots, n$,
\begin{equation*}
\sigma_k(D^2 u) = \sigma_k(\lambda(D^2 u)) = \sum _{1 \le i_1 < i_2 <\cdots<i_k\leq n}\lambda_{i_1}\lambda_{i_2}\cdots\lambda_{i_k},
\end{equation*}
with $\lambda(D^2 u) =(\lambda_1,\cdots,\lambda_n)$ be the eigenvalues of $D^2 u$. We also set $\sigma_0=1$.
Recall that the Garding's cone is defined as
\begin{equation*}
\Gamma_k  = \{ \lambda  \in \mathbb{R}^n :\sigma _i (\lambda ) > 0,\forall 1 \le i \le k\}.
\end{equation*}
If $\lambda(D^2 u) \in \Gamma_k$ for any $x \in \Omega$, then the equation \eqref{1.1} is elliptic (see \cite{L96}), and we say $u$ is a $k$-admissible solution of \eqref{1.1}.

If $l=0$, \eqref{1.1} is known as the $k$-Hessian equation. In particular, \eqref{1.1} is the Laplace equation if $k =1, l=0$, and the Monge-Amp\`{e}re equation if $k =n, l=0$. Hessian quotient equations are a more general form of $k$-Hessian equations, which appear naturally in classical geometry, conformal geometry and K\"{a}hler geometry, etc.

For the Dirichlet problem of elliptic equaions in $\mathbb{R}^n$, many results are known. For example, the Dirichlet problem of Laplace equation was studied in \cite{GT}, Caffarelli-Nirenberg-Spruck \cite{CNS84} and Ivochkina \cite{I87} solved the Dirichlet problem of the Monge-Amp\`{e}re equation, and Caffarelli-Nirenberg-Spruck \cite{CNS85} solved the Dirichlet problem of the $k$-Hessian equation. For the general Hessian quotient equation, the Dirichlet problem was solved by Trudinger in \cite{T95}.

Also, the Neumann or oblique derivative problem of partial differential equations
was widely studied. For a priori estimates and the existence theorem of Laplace equation
with Neumann boundary condition, we refer to the book \cite{GT}. Also, we can see the recent book written by Lieberman
\cite{L13} for the Neumann or oblique derivative problem of linear and quasilinear elliptic equations. In 1986, Lions-Trudinger-Urbas solved the Neumann problem of Monge-Amp\`{e}re equation in the celebrated paper \cite{LTU86}. For related results on the Neumann or oblique derivative problem for some class fully
nonlinear elliptic equations can be found in Urbas \cite{U95} and \cite{U96}. For the the Neumann problem of $k$-Hessian equations,  Trudinger \cite{T87}
established the existence theorem when the domain is a ball, and he conjectured (in \cite{T87}, page 305) that one could solve the problem in sufficiently smooth strictly convex domains. Recently, Ma-Qiu \cite{MQ15} gave a positive answer to this problem and solved the the Neumann problem of $k$-Hessian equations in strictly convex domains.

Naturally, we want to know how about the Neumann problem of Hessian quotient equations. In this paper, we establish global $C^2$ estimates of the Neumann problem of Hessian quotient equations and obtain the existence theorem as follows,
\begin{theorem} \label{th1.1}
Suppose that $\Omega \subset \mathbb{R}^n$ is a $C^4$ convex and strictly $(k-1)$-convex domain, $0 \leq l < k \leq n$, $\nu$ is the outer unit normal vector of $\partial \Omega$, $f \in C^2(\overline{\Omega})$ is a positive function and $\varphi \in C^3(\partial \Omega)$. Then there exists a unique $k$-admissible solution $u \in C^{3, \alpha}(\overline \Omega)$ of the Neumann problem of Hessian quotient equation
\begin{align} \label{1.2}
\left\{ \begin{array}{l}
\frac{\sigma _k (D^2 u)}{\sigma _l (D^2 u)} = f(x),  \quad \text{in} \quad \Omega,\\
u_\nu = -u + \varphi(x),\qquad \text{on} \quad \partial \Omega.
 \end{array} \right.
\end{align}
\end{theorem}

\begin{remark}
The $C^2$ domain $\Omega \subset \mathbb{R}^n$ is convex, that is, $\kappa_i(x) \geq 0$ for any $x \in \partial \Omega$ and $i =1, \cdots, n-1$, or equivalently, $\kappa (x) = (\kappa_1, \cdots, \kappa_{n-1}) \in \overline{\Gamma_{n-1}}$ for any $x \in \partial \Omega$, where $\kappa (x) = (\kappa_1, \cdots, \kappa_{n-1})$ denote the principal curvatures of $\partial \Omega$ with respect to its inner normal. Similarly,
$\Omega$ is strictly $(k-1)$-convex, in the sense of $\kappa (x) = (\kappa_1, \cdots, \kappa_{n-1}) \in \Gamma_{k-1}$ for any $x \in \partial \Omega$. For simplify, a domain is called strictly convex if it is strictly $(n-1)$-convex.
\end{remark}

Following the idea in \cite{QX16}, we can obtain the existence theorem of the classical Neumann problem of Hessian quotient equations
\begin{theorem}\label{th1.2}
Suppose that $\Omega \subset \mathbb{R}^n$ is a $C^4$ strictly convex domain, $0 \leq l < k \leq n$, $\nu$ is the outer unit normal vector of $\partial \Omega$, $f \in C^2(\overline{\Omega})$ is a positive function and $\varphi \in C^3(\partial \Omega)$. Then there exists a unique constant $c$, such that the Neumann problem of the Hessian quotient equation
\begin{align} \label{1.3}
\left\{ \begin{array}{l}
\frac{\sigma _k (D^2 u)}{\sigma _l (D^2 u)} = f(x),  \quad \text{in} \quad \Omega,\\
u_\nu = c + \varphi(x),\qquad \text{on} \quad \partial \Omega,
 \end{array} \right.
\end{align}
has $k$-admissible solutions $u \in C^{3, \alpha}(\overline \Omega)$, which are unique up to a constant.
\end{theorem}

\begin{remark} \label{rmk1.4}
For the classical Neumann problem of Hessian quotient equations \eqref{1.3}, it is easy to know that a solution plus any constant is still a solution. So we cannot obtain a uniform bound for the solutions of \eqref{1.3}, and cannot use the method of continuity directly to get the existence. As in Lions-Trudinger-Urbas \cite{LTU86} and Qiu-Xia \cite{QX16}, we consider the $k$-admissible solution $u^\varepsilon$ of the equation
\begin{align}\label{1.4}
\left\{ \begin{array}{l}
\frac{\sigma _k (D^2 u^\varepsilon)}{\sigma _l (D^2 u^\varepsilon)} = f(x),  \qquad \qquad \text{in} \quad \Omega,\\
(u^\varepsilon)_\nu = - \varepsilon u^\varepsilon  + \varphi(x),\qquad \text{on} \quad \partial \Omega,
 \end{array} \right.
\end{align}
for any small $\varepsilon >0$. We need to establish a priori estimates of $u^\varepsilon$ independent of $\varepsilon$, and the strict convexity of $\Omega$ plays an important role. By letting $\varepsilon \rightarrow 0$ and a perturbation argument, we can obtain a solution of \eqref{1.3}. The uniqueness holds from the maximum principle and Hopf Lemma.
\end{remark}

\begin{remark}
In the recent papers \cite{JT15, JT16}, Jiang and Trudinger studied the general oblique boundary value problems for augmented Hessian equations with some regular conditions and some concavity conditions. But here, the problems \eqref{1.2} and \eqref{1.3} do not satisfy the strictly regular condition and the uniform concavity condition.
\end{remark}

\begin{remark}
As we all know, the Dirichlet problems of Hessian and Hessian quotient equations are solved in strictly $(k-1)$-convex domains. For the Neumann problems, we also want to know the existence results in strictly $(k-1)$-convex but not convex domains. A special case, that is
\begin{align} \label{1.5}
\left\{ \begin{array}{l}
\sigma _k (D^2 u) = f(x),  \quad \text{in} \quad \Omega \subset \mathbb{R}^n,\\
u_\nu = -u,\qquad \text{on} \quad \partial \Omega,
 \end{array} \right.
\end{align}
is solvable in strictly $(k-1)$-convex domains, and see \cite{Q16} for the proof. For general cases, the problem is open.
\end{remark}

The rest of the paper is organized as follows. In Section 2, we collect some properties of the lementary symmetric function $\sigma_k$, and establish some key inequalities of Hessian quotient operators. Following the idea of Lions-Trudinger-Urbas \cite{LTU86} and Ma-Qiu \cite{MQ15}, we establish the $C^0$, $C^1$ and $C^2$ estimates for the Neumann problem of Hessian quotient equations in Section 3, Section 4, Section 5, respectively. At last, we prove Theorem \ref{th1.1} and Theorem \ref{th1.2} in Section 6.

\section{preliminary}

In this section, we give some basic properties
of elementary symmetric functions, which could be found in
\cite{L96}, and establish some key inequalities of Hessian quotient operators.

\subsection{Basic properties of elementary symmetric functions}

First, we denote by $\sigma _k (\lambda \left| i \right.)$ the symmetric
function with $\lambda_i = 0$ and $\sigma _k (\lambda \left| ij
\right.)$ the symmetric function with $\lambda_i =\lambda_j = 0$.
\begin{proposition}\label{prop2.1}
Let $\lambda=(\lambda_1,\dots,\lambda_n)\in\mathbb{R}^n$ and $k
= 1, \cdots,n$, then
\begin{align*}
&\sigma_k(\lambda)=\sigma_k(\lambda|i)+\lambda_i\sigma_{k-1}(\lambda|i), \quad \forall \,1\leq i\leq n,\\
&\sum_i \lambda_i\sigma_{k-1}(\lambda|i)=k\sigma_{k}(\lambda),\\
&\sum_i\sigma_{k}(\lambda|i)=(n-k)\sigma_{k}(\lambda).
\end{align*}
\end{proposition}

We also denote by $\sigma _k (W \left|
i \right.)$ the symmetric function with $W$ deleting the $i$-row and
$i$-column and $\sigma _k (W \left| ij \right.)$ the symmetric
function with $W$ deleting the $i,j$-rows and $i,j$-columns. Then
we have the following identities.
\begin{proposition}\label{prop2.2}
Suppose $W=(W_{ij})$ is diagonal, and $m$ is a positive integer,
then
\begin{align*}
\frac{{\partial \sigma _m (W)}} {{\partial W_{ij} }} = \begin{cases}
\sigma _{m - 1} (W\left| i \right.), &\text{if } i = j, \\
0, &\text{if } i \ne j.
\end{cases}
\end{align*}
\end{proposition}

Recall that the Garding's cone is defined as
\begin{equation}\label{2.1}
\Gamma_k  = \{ \lambda  \in \mathbb{R}^n :\sigma _i (\lambda ) >
0,\forall 1 \le i \le k\}.
\end{equation}

\begin{proposition}\label{prop2.3}
Let $\lambda \in \Gamma_k$ and $k \in \{1,2, \cdots, n\}$. Suppose that
$$
\lambda_1 \geq \cdots \geq \lambda_k \geq \cdots \geq \lambda_n,
$$
then we have
\begin{align}
\label{2.2}& \sigma_{k-1} (\lambda|n) \geq \sigma_{k-1} (\lambda|n-1) \geq \cdots \geq \sigma_{k-1} (\lambda|k) \geq \cdots \geq \sigma_{k-1} (\lambda|1) >0; \\
\label{2.3}& \lambda_1 \geq \cdots \geq \lambda_k >  0, \quad \sigma _k (\lambda)\leq C_n^k  \lambda_1 \cdots \lambda_k; \\
\label{2.4}& \lambda _1 \sigma _{k - 1} (\lambda |1) \geq \frac{k} {{n}}\sigma _k (\lambda),
\end{align}
where $C_n^k = \frac{n!}{k! (n-k)!}$.
\end{proposition}

\begin{proof}
All the properties are well known. For example, see \cite{L96} or \cite{HS99} for a proof of \eqref{2.2},
\cite{L91} for \eqref{2.3}, and \cite{CW01} or \cite{HMW10} for \eqref{2.4}.
\end{proof}

The generalized Newton-MacLaurin inequality is as follows, which will be used all the time.
\begin{proposition}\label{prop2.4}
For $\lambda \in \Gamma_k$ and $k > l \geq 0$, $ r > s \geq 0$, $k \geq r$, $l \geq s$, we have
\begin{align} \label{2.5}
\Bigg[\frac{{\sigma _k (\lambda )}/{C_n^k }}{{\sigma _l (\lambda )}/{C_n^l }}\Bigg]^{\frac{1}{k-l}}
\le \Bigg[\frac{{\sigma _r (\lambda )}/{C_n^r }}{{\sigma _s (\lambda )}/{C_n^s }}\Bigg]^{\frac{1}{r-s}}.
\end{align}
\end{proposition}
\begin{proof}
See \cite{S05}.
\end{proof}

\subsection{Key Lemmas}

In the establishment of the a priori estimates, the following inequalities of Hessian quotient operators play an important role.

\begin{lemma} \label{lem2.5}
Suppose $\lambda = (\lambda_1, \lambda_2, \cdots, \lambda_n) \in \Gamma_k$, $k \geq 1$, and $ \lambda_1 <0$. Then we have
\begin{align} \label{2.6}
\sigma_{m} (\lambda |1)  \geq   \sigma_{m} (\lambda), \quad \forall \quad m =0, 1, \cdots, k.
\end{align}
Moreover, we have
\begin{align} \label{2.7}
\frac{{\partial [\frac{{\sigma _k (\lambda )}}{{\sigma _l (\lambda )}}]}}
{{\partial \lambda _1 }} \geqslant \frac{n}{k}\frac{{k - l}}{{n - l}}\frac{1}
{{n - k + 1}}\sum\limits_{i = 1}^n {\frac{{\partial [\frac{{\sigma _k (\lambda )}}
{{\sigma _l (\lambda )}}]}}{{\partial \lambda _i }}}, \quad \forall \quad 0 \leq l < k.
\end{align}
\end{lemma}

\begin{proof}
Firstly, we can easily get
\[
\sigma _m (\lambda ) = \sigma _m (\lambda |1) + \lambda _1 \sigma _{m - 1} (\lambda |1) \leqslant \sigma _m (\lambda |1),
\]
so \eqref{2.6} holds.

Directly calculations yield
\begin{align} \label{2.8}
\sum\limits_{i = 1}^n {\frac{{\partial [\frac{{\sigma _k (\lambda )}}
{{\sigma _l (\lambda )}}]}}{{\partial \lambda _i }}}  =& \sum\limits_{i = 1}^n {\frac{{\sigma _{k - 1} (\lambda |i)\sigma _l (\lambda ) - \sigma _k (\lambda )\sigma _{l - 1} (\lambda |i)}} {{\sigma _l (\lambda )^2 }}}  \notag \\
=& \frac{{(n - k + 1)\sigma _{k - 1} (\lambda )\sigma _l (\lambda ) - (n - l + 1)\sigma _k (\lambda )\sigma _{l - 1} (\lambda )}}
{{\sigma _l (\lambda )^2 }} \notag \\
\leq& (n - k + 1)\frac{{\sigma _{k - 1} (\lambda )\sigma _l (\lambda )}} {{\sigma _l (\lambda )^2 }},
\end{align}
hence we can get
\begin{align} \label{2.9}
  \frac{{\partial [\frac{{\sigma _k (\lambda )}} {{\sigma _l (\lambda )}}]}}
{{\partial \lambda _1 }} =& \frac{{\sigma _{k - 1} (\lambda |1)\sigma _l (\lambda ) - \sigma _k (\lambda )\sigma _{l - 1} (\lambda |1)}}
{{\sigma _l (\lambda )^2 }} \notag \\
=& \frac{{\sigma _{k - 1} (\lambda |1)\sigma _l (\lambda |1) - \sigma _k (\lambda |1)\sigma _{l - 1} (\lambda |1)}}
{{\sigma _l (\lambda )^2 }} \notag \\
\geq& (1 - \frac{l} {k}\frac{{n - k}} {{n - l}})\frac{{\sigma _{k - 1} (\lambda |1)\sigma _l (\lambda |1)}}
{{\sigma _l (\lambda )^2 }} \notag \\
\geq& \frac{n} {k}\frac{{k - l}} {{n - l}}\frac{{\sigma _{k - 1} (\lambda )\sigma _l (\lambda )}}
{{\sigma _l (\lambda )^2 }} \notag \\
\geq& \frac{n} {k}\frac{{k - l}}{{n - l}}\frac{1} {{n - k + 1}}\sum\limits_{i = 1}^n {\frac{{\partial [\frac{{\sigma _k (\lambda )}}
{{\sigma _l (\lambda )}}]}} {{\partial \lambda _i }}}.
\end{align}

\end{proof}

\begin{lemma} \label{lem2.6}
Suppose $A=\{ a_{ij}\}_{n \times n}$ satisfies
 \begin{align} \label{2.10}
a_{11} <0, \quad  \{ a_{ij}\}_{2 \leq i,j \leq n} \quad \text{ is diagonal},
 \end{align}
and $\lambda (A) \in \Gamma_k$ with $k \geq 1$. Then we have
\begin{align} \label{2.11}
\frac{{\partial [\frac{{\sigma _k (A )}}{{\sigma _l (A)}}]}}
{{\partial a_{11} }} \geqslant \frac{n}{k}\frac{{k - l}}{{n - l}}\frac{1}
{{n - k + 1}}\sum\limits_{i = 1}^n {\frac{{\partial [\frac{{\sigma _k (A )}}
{{\sigma _l (A)}}]}}{{\partial a_{ii} }}}, \quad \forall \quad 0 \leq l < k,
\end{align}
and
\begin{align} \label{2.12}
\sum\limits_{i = 1}^n {\frac{{\partial [\frac{{\sigma _k (A )}}
{{\sigma _l (A)}}]}}{{\partial a_{ii} }}}  \geq& \frac{{k - l}}{k} \frac{1}{C_n^l} (-a_{11})^{k-l-1}, \quad \forall \quad 0 \leq l < k.
\end{align}
\end{lemma}

\begin{proof}
See Lemma 3.9 in \cite{C15} for the proof of \eqref{2.11}.

To prove \eqref{2.12}, we assume that $\lambda = (\lambda_1, \lambda_2, \cdots, \lambda_n)$ are the eigenvalues of $A$, and $\lambda_1 \geq \lambda_2 \geq  \cdots \geq \lambda_n$. It is easy to know that $\lambda_n \leq a_{11} <0$. Direct calculation yields
\begin{align} \label{2.13}
\sum\limits_{i = 1}^n {\frac{{\partial [\frac{{\sigma _k (A )}} {{\sigma _l (A)}}]}}{{\partial a_{ii} }}} =& \sum\limits_{i = 1}^n {\frac{{\partial [\frac{{\sigma _k (\lambda )}} {{\sigma _l (\lambda)}}]}}{{\partial \lambda_{i} }}} \ge \frac{{k - l}}{k}(n - k + 1) \frac{{\sigma _{k - 1} (\lambda )}}{{\sigma _l (\lambda )}} \notag \\
\geq& \frac{{k - l}}{k}(n - k + 1) \frac{{ \lambda_1 \cdots \lambda_l \sigma _{k -l- 1} (\lambda|1, \cdots, l )}}{{C_n^l \lambda_1 \cdots \lambda_l}} \notag \\
=& \frac{{k - l}}{k}(n - k + 1) \frac{1}{C_n^l}  \sigma _{k -l- 1} (\lambda|1, \cdots, l )\notag \\
=& \frac{{k - l}}{k} \frac{1}{C_n^l}  \Big[\sum_{j=l+1}^{n-1} \sigma _{k -l- 1} (\lambda|1, \cdots, l, j)+ \sigma _{k -l- 1} (\lambda|1, \cdots, l, n)\Big]\notag \\
\geq& \frac{{k - l}}{k} \frac{1}{C_n^l}  \sigma _{k -l- 1} (\lambda|1, \cdots, l, n)\notag \\
\geq& \frac{{k - l}}{k} \frac{1}{C_n^l} (-\lambda_n )^{k-l-1}\geq \frac{{k - l}}{k} \frac{1}{C_n^l} (-a_{11})^{k-l-1}.
\end{align}
Hence \eqref{2.12} holds.

\end{proof}

\begin{lemma} \label{lem2.7}
Suppose $\lambda = (\lambda_1, \lambda_2, \cdots, \lambda_n) \in \Gamma_k$, $k \geq 2$, and $ \lambda_2 \geq \cdots \geq \lambda_n$. If $\lambda_1 > 0$, $\lambda_n < 0$, $\lambda_1 \geq \delta \lambda_2$, and $- \lambda_n \geq \varepsilon \lambda_1$ for small positive constants $\delta$ and $\varepsilon$,  then we have
\begin{align} \label{2.14}
\sigma_{m} (\lambda |1)  \geq  c_0 \sigma_{m} (\lambda), \quad \forall \quad m =0, 1, \cdots, k-1,
\end{align}
where $c_0 = \min \{ \frac{{\varepsilon ^2 \delta ^2}}{{2(n - 2)(n - 1)}},\frac{{\varepsilon ^2 \delta }}{{4(n- 1)}}\}$. Moreover, we have
\begin{align}\label{2.15}
\frac{{\partial [\frac{{\sigma _k (\lambda )}}{{\sigma _l (\lambda )}}]}}
{{\partial \lambda _1 }} \geqslant c_1\sum\limits_{i = 1}^n {\frac{{\partial [\frac{{\sigma _k (\lambda )}}
{{\sigma _l (\lambda )}}]}}{{\partial \lambda _i }}}, \quad \forall \quad 0 \leq l < k,
\end{align}
where $c_1 = \frac{n} {k}\frac{{k - l}}{{n - l}}\frac{c_0^2} {{n - k + 1}}$.
\end{lemma}

\begin{proof}
The idea of proof of \eqref{2.14} is from \cite{MQ15}, and we produce the proof here.

For $m=0$, \eqref{2.14} holds directly. In the following, we assume $1 \leq m \leq k-1$.

Firstly,  if $\lambda _1  \geqslant \lambda _2$, we have from \eqref{2.4}
\begin{align}\label{2.16}
\lambda _1 \sigma _{m - 1} (\lambda |1n) \geq \frac{m} {{n - 1}}\sigma _m (\lambda |n).
\end{align}
If $\lambda _1  < \lambda _2$, we have
\begin{align}\label{2.17}
\lambda _1 \sigma _{m - 1} (\lambda |1n) \geq& \lambda _1 \sigma _{m - 1} (\lambda |2n) \geqslant \delta \lambda _2 \sigma _{m - 1} (\lambda |2n) \notag \\
\geq& \delta \frac{m}{{n - 1}}\sigma _m (\lambda |n).
\end{align}
Hence from \eqref{2.16} and \eqref{2.17}, it holds
\begin{align}\label{2.18}
( - \lambda _n )\sigma _{m - 1} (\lambda |1n) \geq& \varepsilon \lambda _1 \sigma _{m - 1} (\lambda |1n) \notag \\
\geq& \varepsilon \delta \frac{m} {{n - 1}}\sigma _m (\lambda |n) \notag \\
\geq& \varepsilon \delta \frac{m} {{n - 1}}\sigma _m (\lambda ).
\end{align}

Let
\begin{align}\label{2.19}
  \theta  = \frac{{\varepsilon \delta }} {{2(n - 2)}}.
\end{align}
We divide into two cases to prove \eqref{2.14}.

$\blacklozenge$ Case 1: $\sigma _m (\lambda |1) \geq \theta ( - \lambda _n )\sigma _{m-1} (\lambda |1n)$.

In this case, we can get directly from \eqref{2.18}
\begin{align}\label{2.20}
\sigma _m (\lambda |1) \geq& \theta ( - \lambda _n )\sigma _{m-1} (\lambda |1n) \notag \\
\geq& \theta \varepsilon \delta \frac{m} {{n - 1}}\sigma _m (\lambda ) \notag \\
=& \frac{{\varepsilon ^2 \delta ^2 m}} {{2(n - 2)(n - 1)}}\sigma _m (\lambda ) \geq \frac{{\varepsilon ^2 \delta ^2 }} {{2(n - 2)(n - 1)}}\sigma _m (\lambda ).
\end{align}

$\blacklozenge$ Case 2: $\sigma _m (\lambda |1) < \theta ( - \lambda _n )\sigma _{m-1} (\lambda |1n)$.

In this case, we have
\begin{align}\label{2.21}
(m + 1)\sigma _{m + 1} (\lambda |1) =& \sum_{ i= 2}^{n}{\lambda_i}\sigma_{m}(\lambda|1i)=
\sum_{ i= 2}^{n}{\lambda_i} [\sigma_{m}(\lambda|1)-\lambda_i \sigma_{m-1}(\lambda|1i)] \notag \\
=& \sum_{ i= 2}^{n}{\lambda_i}\sigma_{m}(\lambda|1)-\sum_{ i= 2}^{n}{\lambda_i^2\sigma_{m-1}(\lambda|1i)} \notag \\
\le& \sum_{ i= 2}^{n}{\lambda_i}\sigma_{m}(\lambda|1)-\lambda_n^2\sigma_{m-1}(\lambda|1n)  \notag  \\
\leq& (n - 2)\lambda _2 \sigma _m (\lambda |1) - \lambda _n ^2 \sigma _{m - 1} (\lambda |1n) \notag \\
<& \frac{{(n - 2)}} {\delta }\lambda _1 \theta ( - \lambda _n )\sigma _{m-1} (\lambda |1n) + \varepsilon \lambda _n \lambda _1 \sigma _{m - 1} (\lambda |1n) \notag \\
=&  - \frac{\varepsilon } {2}\lambda _1 ( - \lambda _n )\sigma _{m-1} (\lambda |1n).
\end{align}
From \eqref{2.18}, we can get
\begin{align}\label{2.22}
(m+ 1)\sigma _{m + 1} (\lambda |1) <  - \frac{{\varepsilon ^2 \delta }} {2}\frac{m}
{{n - 1}}\lambda _1 \sigma _m (\lambda ),
\end{align}
then
\begin{align}\label{2.23}
\sigma _m (\lambda |1) =& \frac{{\sigma _{m + 1} (\lambda ) - \sigma _{m + 1} (\lambda |1)}}
{{\lambda _1 }} > \frac{{ - \sigma _{m + 1} (\lambda |1)}} {{\lambda _1 }} \notag \\
>& \frac{{\varepsilon ^2 \delta }} {{2(m + 1)}}\frac{m} {{n - 1}}\sigma _m (\lambda ) \geq \frac{{\varepsilon ^2 \delta }} {{4(n- 1)}}\sigma _m (\lambda ).
\end{align}
Hence \eqref{2.14} holds.

From \eqref{2.9} and the generalized Newton-MacLaurin inequality, we can get
\begin{align}\label{2.24}
  \frac{{\partial [\frac{{\sigma _k (\lambda )}} {{\sigma _l (\lambda )}}]}}
{{\partial \lambda _1 }} =& \frac{{\sigma _{k - 1} (\lambda |1)\sigma _l (\lambda ) - \sigma _k (\lambda )\sigma _{l - 1} (\lambda |1)}}
{{\sigma _l (\lambda )^2 }} \notag \\
=& \frac{{\sigma _{k - 1} (\lambda |1)\sigma _l (\lambda |1) - \sigma _k (\lambda |1)\sigma _{l - 1} (\lambda |1)}}
{{\sigma _l (\lambda )^2 }} \notag \\
\geq& (1 - \frac{l} {k}\frac{{n - k}} {{n - l}})\frac{{\sigma _{k - 1} (\lambda |1)\sigma _l (\lambda |1)}}
{{\sigma _l (\lambda )^2 }} \notag \\
\geq& \frac{n} {k}\frac{{k - l}} {{n - l}} c_0^2 \frac{{\sigma _{k - 1} (\lambda )\sigma _l (\lambda )}}
{{\sigma _l (\lambda )^2 }} \notag \\
\geq& \frac{n} {k}\frac{{k - l}}{{n - l}}\frac{c_0^2} {{n - k + 1}}\sum\limits_{i = 1}^n {\frac{{\partial [\frac{{\sigma _k (\lambda )}}
{{\sigma _l (\lambda )}}]}} {{\partial \lambda _i }}}.
\end{align}

\end{proof}

\begin{remark}
These lemmas play an important role in the establishment of a priori estimates. Precisely, Lemma \ref{lem2.6} is the key of the gradient estimates in Section 4, including the interior gradient estimate and the near boundary gradient estimate. Lemmas \ref{lem2.5} and Lemma \ref{lem2.7} are the keys of the lower and upper estimates of double normal second order derivatives on the boundary in Section 5, respectively.
\end{remark}

\section{$C^0$ estimate}

The $C^0$ estimate is easy. For completeness, we produce a proof here following the idea of Lions-Trudinger-Urbas \cite{LTU86} and Ma-Qiu \cite{MQ15}.
\begin{theorem} \label{th3.1}
Suppose $\Omega \subset \mathbb{R}^n$ is a $C^1$ bounded domain, $f \in C^0(\overline{\Omega})$ is a positive function, $\varphi \in C^0(\partial \Omega)$ and $u \in C^2(\Omega)\cap C^1(\overline \Omega)$ is the $k$-admissible solution of Hessian quotient equation \eqref{1.2}, then we have
\begin{align}\label{3.1}
\sup_\Omega |u|  \leq M_0,
\end{align}
where $M_0$ depends on $n$, $k$, $l$, $diam(\Omega)$, $\max\limits_{\partial \Omega} |\varphi|$ and $\sup\limits_\Omega f$.
\end{theorem}

\begin{proof}
Firstly, since $u$ is subharmonic, the maximum of $u$ is attained at some boundary point $x_0 \in \partial \Omega$. Then we can get
\begin{align}\label{3.2}
0 \leq u_\nu (x_0) =-u(x_0) + \varphi (x_0).
\end{align}
Hence
\begin{align}\label{3.3}
\max_\Omega u = u(x_0) \leq \varphi (x_0) \leq \max_{\partial \Omega} |\varphi|.
\end{align}

For a fixed point $x_1 \in \Omega$,  and a constant $A = \frac{1}{2} [\frac{C_n^l}{C_n^k} \sup_\Omega f]^{\frac{1}{k-l}}$, we have
\begin{align} \label{3.4}
\frac{\sigma _k (D^2 u)}{\sigma _l (D^2 u)} = f(x) \leq \sup_\Omega f = \frac{\sigma _k (D^2 (A|x-x_1|^2))}{\sigma _l (D^2(A|x-x_1|^2))}.
\end{align}
By the comparison principle, we know $u - A|x-x_1|^2$ attains its minimum at some boundary point $x_2 \in \partial \Omega$. Then
\begin{align}\label{3.5}
0 \geq& (u - A|x-x_1|^2)_\nu (x_2) = u_\nu (x_2) -2A (x_2-x_1)\cdot \nu  \notag \\
=&-u(x_2) + \varphi (x_2)-2A (x_2-x_1)\cdot \nu  \notag \\
\geq& -u(x_2) - \max_{\partial \Omega} |\varphi| - 2A \text{diam}(\Omega).
\end{align}
Hence
\begin{align}\label{3.6}
\min_\Omega u \geq \min_\Omega (u - A|x-x_1|^2) =& u (x_2) - A|x_2-x_1|^2 \notag \\
\geq& - \max_{\partial \Omega} |\varphi| - 2A \text{diam}(\Omega)- A \text{diam}(\Omega)^2.
\end{align}
\end{proof}

Here, following the proof of Theorem \ref{th3.1}, we can easily obtain
\begin{theorem} \label{th3.2}
Suppose $\Omega \subset \mathbb{R}^n$ is a $C^1$ bounded domain, $f \in C^0(\overline{\Omega})$ is a positive function, $\varphi \in C^0(\partial \Omega)$ and $u^\varepsilon \in C^2(\Omega)\cap C^1(\overline \Omega)$ is the $k$-admissible solution of the Hessian quotient equation \eqref{1.4} with $\varepsilon \in (0,1)$, then we have
\begin{align}\label{3.7}
\sup_\Omega |\varepsilon u^\varepsilon|  \leq \overline{M_0},
\end{align}
where $\overline{M_0}$ depends on $n$, $k$, $l$, $ \text{diam} (\Omega)$, $\max\limits_{\partial \Omega} |\varphi|$ and $\sup\limits_\Omega f$.
\end{theorem}

\section{Global gradient estimate}

In this section, we prove the global gradient estimate, involving the interior gradient estimate and the near boundary gradient estimate.  To state our theorems, we denote $d(x) = \textrm{dist}(x, \partial \Omega)$, and $\Omega_{\mu} = \{ x \in \Omega| d(x) < \mu \}$ where $\mu$ is a small positive universal constant depending only on $\Omega$. In Subsection 4.1, we give the interior gradient estimate in $\Omega \setminus \Omega_{\mu}$, and in Subsection 4.2 we establish the near boundary gradient estimate in $\Omega_{\mu}$, following the idea of Ma-Qiu-Xu \cite{MQX16} and Ma-Qiu \cite{MQ15}.

\subsection{Interior gradient estimate}

The interior gradient estimate is established in \cite{C15} as follows
\begin{theorem} \label{th4.1}
Suppose $u \in C^3 (B_r(0))$ is a $k$-admissible solution to the Hessian quotient equation
\begin{align}\label{4.1}
\frac{\sigma _k (D^2 u)}{\sigma _l (D^2 u)} = f(x, u, D u),  \quad x \in B_r(0) \subset \mathbb{R}^n,
\end{align}
with $f(x, u, D u) >0$ in $ B_r(0)$ and $f(x, u, p) \in C^{1}( B_r(0) \times \mathbb{R} \times  \mathbb{R}^n)$. Then
\begin{align}\label{4.2}
|D u (0)|  \leq C \Big( \frac{\mathop {osc}\limits_{B_r (0)} u }{r}+ [\mathop {osc}\limits_{B_r (0)} u]^{\frac{k-l+1}{2(k-l)}}  + [\mathop {osc}\limits_{B_r (0)} u]^{\frac{k-l}{2(k-l)+1}}   \Big),
\end{align}
where $C$ is a positive constant depending only on $n$, $k$, $l$ and $|D_x f|_{C^0}$, $|D_u f|_{C^0}$, $|D_p f|_{C^0}$.
\end{theorem}

Hence we can get the interior gradient estimate in $\Omega \setminus \Omega_{\mu}$ directly.
\begin{theorem} \label{th4.2}
Suppose $\Omega \subset \mathbb{R}^n$ is a bounded domain, $f \in C^1(\overline{\Omega})$ is a positive function and $u \in C^3(\Omega)$ is a $k$-admissible solution of Hessian quotient equation \eqref{1.1}, then we have
\begin{align}\label{4.3}
\sup_{\Omega \setminus \Omega_{\mu}} |D u|  \leq M_1,
\end{align}
where $M_1$ depends on $n$, $k$, $l$, $\mu$, $|u|_{C^0}$ and $|D_x f|_{C^0}$.
\end{theorem}

\subsection{ Near boundary gradient estimate}

\begin{theorem} \label{th4.3}
Suppose $\Omega \subset \mathbb{R}^n$ is a $C^3$ bounded domain, $f \in C^1(\overline{\Omega})$ is a positive function, $\varphi \in C^3 (\partial \Omega)$ and $u \in C^3(\Omega) \cap C^2(\overline \Omega)$ is the $k$-admissible solution of the Hessian quotient equation \eqref{1.2}, then we have
\begin{align}\label{4.4}
\sup_{\Omega_{\mu}} |D u|  \leq \max\{M_1, \widetilde{M_1}\},
\end{align}
where $M_1$ depends on $n$, $k$, $l$, $\mu$, $M_0$ and $|D f|_{C^0}$, and $\widetilde{M_1}$ depends on $n$, $k$, $l$, $\mu$, $\Omega$, $M_0$, $\sup f$, $|D f|_{C^0}$ and $|\varphi|_{C^3}$.
\end{theorem}

\begin{proof}
The proof follows the idea of Ma-Qiu-Xu \cite{MQX16} and Ma-Qiu \cite{MQ15}.

Since $\Omega$ is a $C^3$ domain, it is well known that there exists a small positive universal constant $0<{\mu} < \frac{1}{10}$ such that
$d(x) \in C^3(\overline{\Omega_{\mu}})$. As in Simon-Spruck \cite{SS76} or Lieberman \cite{L13} (in page 331), we can
extend $\nu$ by $\nu = - D d$ in $\Omega_{\mu}$ and thus $\nu$ is a $C^2(\overline{\Omega_{\mu}})$ vector field. As mentioned
in the book \cite{L13}, we also have the following formulas
\begin{align}
\label{4.5}&|D \nu| + |D^2 \nu| \leq C_0, \quad \text{in} \quad \overline{\Omega_{ \mu}}, \\
\label{4.6}&\sum_{i=1}^n \nu^i D_j \nu^i =0, \quad \sum_{i=1}^n \nu^i D_i \nu^j =0, \quad |\nu| =1,  \quad \text{in} \quad \overline{\Omega_{\mu}},
\end{align}
where $C_0$ depends only on $n$ and $\Omega$. As in \cite{L13}, we define
\begin{align}\label{4.7}
c^{ij} = \delta_{ij} - \nu^i \nu^j,  \quad \text{in} \quad \overline{\Omega_{ \mu}},
\end{align}
and for a vector $\zeta \in \mathbb{R}^n$, we write $\zeta'$ for the vector with $i$-th component $\sum_{j=1}^n c^{ij} \zeta^j$.
Then we have
\begin{align}\label{4.8}
|(D u)'|^2 = \sum_{i,j=1}^n c^{ij} u_i u_j,  \quad \text{and} \quad |D u|^2 =|(D u)'|^2 + u_\nu ^2.
\end{align}

We consider the auxiliary function
\begin{align} \label{4.9}
G(x)  = \log |D w|^2  +h(u) + g(d),
\end{align}
where
\begin{align*}
&w(x) = (1-d)u + \varphi(x) d(x),  \\
&h(u) = -\log(1+M_0 -u),\\
&g(d)= \alpha_0 d,
\end{align*}
with $\alpha_0 > 0 $ to be determined later. Note that here $\varphi \in C^3(\overline \Omega)$ is a extension with universal $C^3$ norms.

It is easy to know $G(x)$ is well-defined in $\overline{\Omega_{\mu}}$. Then we assume that $G(x)$ attains its maximum at a point $x_0 \in \overline{\Omega_{\mu}}$. If we have $|D u|(x_0) \leq 10n[|\varphi|_{C^1({\Omega})} + \sup_\Omega |u|]$, then we can get directly
\begin{align}\label{4.10}
\sup_{\Omega_{\mu}} \log|D u|^2 =&\sup_{\Omega_{\mu}} \log \frac{|D w- d D \varphi+ (\varphi-u) D d|^2}{(1-d)^2}\notag \\
\leq& 2 [\sup_{\Omega_{\mu}} \log |D w|^2 + (|\varphi|_{C^1({\Omega})} + \sup_\Omega |u|)^2] \notag \\
\leq& 2 [\log |D w|^2(x_0) + \sup_{\Omega_{\mu}} |h(u)| + \sup_{\Omega_{\mu}} |g(d)| + (|\varphi|_{C^1({\Omega})} + \sup_\Omega |u|)^2] \notag \\
\leq& 2 [\log |D u|^2(x_0) + \log(1+2M_0 ) + \alpha_0 + 2(|\varphi|_{C^1({\Omega})} + \sup_\Omega |u|)^2] \notag \\
\leq& 2 [\log(1+2M_0 ) + \alpha_0 + (10n+2)(|\varphi|_{C^1({\Omega})} + \sup_\Omega |u|)^2].
\end{align}
So \eqref{4.4} holds.

Hence, we can assume $|D u|(x_0) > 10n[|\varphi|_{C^1({\Omega})} + \sup_\Omega |u|]$ in the following. Then we have
\begin{align}\label{4.11}
 \frac{1}{2} |D u|^2 \leq |D w|^2 \leq 2 |D u|^2,
 \end{align}
where we used $w_i = (1-d)u_i + \varphi_i d + (\varphi-u) d_i$. Now we divide into three cases to complete the proof of Theorem \ref{th4.3}.

$\blacklozenge$ CASE I: $x_0 \in \partial \Omega_\mu \cap \Omega$.

Then $x_0 \in \Omega \setminus \Omega_\mu$, and  we can use the interior gradient estimate, that is from Theorem \ref{th4.2},
\begin{align}\label{4.12}
|D u| (x_0) \leq \sup_{\Omega \setminus \Omega_{\mu}} |D u|  \leq M_1,
\end{align}
then we can prove \eqref{4.4} by a calculation similar with \eqref{4.10}.

$\blacklozenge$ CASE II: $x_0 \in \partial \Omega$.

At $x_0$, we have $d=0$, and
\begin{align}\label{4.13}
0 \le G_\nu   = \frac{{(|D w|^2 )_i \nu ^i }}{{|D w|^2 }} + h'u_\nu   - g'.
\end{align}
We know from \eqref{4.8}
\begin{align}
|D w|^2  = |(D w)'|^2  + w_\nu  ^2 = c^{pq} w_p w_q  + w_\nu  ^2,  \notag
\end{align}
and by the Neumann boundary condition, we can get
\begin{align}
w_\nu   = (1 - d)u_\nu   + [u - \varphi ] + \varphi _\nu  d = 0. \notag
\end{align}
Hence
\begin{align}\label{4.14}
 (|D w|^2 )_i \nu ^i  =& [c^{pq} _{,i} w_p w_q  + 2c^{pq} w_{pi} w_q  + 2w_\nu  D _i w_\nu  ]\nu ^i  \notag \\
  =& c^{pq} _{,i} w_p w_q \nu ^i  + 2c^{pq} [u_{pi}  + (\varphi _p  - u_p )d_i  + (\varphi _i  - u_i )d_p  + (\varphi  - u)d_{ij} ]w_q \nu ^i \notag \\
  \le& C_1 |D w|^2  + 2c^{pq} u_{pi} w_q \nu ^i  + C_2 [|D w| + |D w|^2 ],
\end{align}
where $C_1 = \sum\limits_{pq}|D c^{pq}|_{C^0}$ and $C_2= 2\sum\limits_{pq}| c^{pq}|_{C^0}\big[2|D \varphi|_{C^0} +(|\varphi|_{C^0} + |u|_{C^0})| |D^2 d|_{C^0} +4\big]$. Also by the Neumann boundary condition, we can get
\begin{align}
c^{pq} D _p (u_i \nu ^i ) = c^{pq} D _p [ - u + \varphi ], \notag
\end{align}
so
\begin{align}
c^{pq} u_{pi} \nu ^i  =  - c^{pq} u_i \nu ^i _{, p}  + c^{pq} ( - u_p  + \varphi _p ). \notag
\end{align}
Hence
\begin{align}\label{4.15}
2c^{pq} u_{pi} w_q \nu ^i  =  - 2c^{pq} u_i w_q \nu ^i _{,p}  + c^{pq} ( - u_p  + \varphi _p )w_q  \le C_3 [|D w| + |D w|^2 ],
\end{align}
where $C_3= 4\sum\limits_{pq}| c^{pq}|_{C^0}\big[|D^2 d|_{C^0}+|D \varphi|_{C^0} +1 \big]$. From \eqref{4.13}, \eqref{4.14} and \eqref{4.15}, we get
\begin{align}\label{4.16}
0 \le G_\nu   \le C_1  + C_2  + C_3  + \frac{{C_2  + C_3 }}{{|D w|}} + \frac{{ - u + \varphi }}{{1 + M_0  - u}} - \alpha _0.
\end{align}
We choose
\begin{align}\label{4.17}
\alpha _0  = C_1  + C_2  + C_3  + |u|_{C^0} +|\varphi |_{C^0} + 1,
\end{align}
then
\begin{align}\label{4.18}
|D w| \le C_2  + C_3.
\end{align}
So we can prove \eqref{4.4} by a calculation similar with \eqref{4.10}, or $x_0$ cannot be at the boundary $\partial \Omega$ by a contradiction discussion.

$\blacklozenge$ CASE III: $x_0 \in \Omega_\mu$.

At $x_0$, we have $0 < d < \mu$, and by rotating the coordinate $e_1, \cdots, e_{n}$, we can assume
\begin{align}\label{4.19}
w_1(x_0) = |D w|(x_0) >0,  \quad \{ u_{ij}(x_0) \}_{2 \leq i, j \leq n} \text{ is diagonal}.
\end{align}
In the following, we denote $\widetilde \lambda = (\widetilde \lambda_2, \cdots, \widetilde\lambda_{n})=(u_{22}(x_0), \cdots, u_{nn}(x_0))$, and all the calculations are at $x_0$. So from the definition of $w$, we know $w_i  = (1 - d)u_i  + [\varphi  - u]d_i  + \varphi _i d$, and by \eqref{4.19} we get
\begin{align}
\label{4.20}& u_1  = \frac{{w_1  - [\varphi  - u]d_1  - \varphi _1 d}}{{1 - d}} >0, \\
\label{4.21}& u_i  = \frac{{ - [\varphi  - u]d_i  - \varphi _i d }}{{1 - d}}, \quad  i \ge 2.
\end{align}
By the assumption $|D u|(x_0) > 10n[|\varphi|_{C^1({\Omega})} + \sup_\Omega |u|]$, we know for $i \ge 2$
\begin{align}\label{4.22}
|u_i|\leq \frac{{ |\varphi | +|u| + |\varphi _i|}}{{1 - d}} \leq \frac{1}{9n}|D u|(x_0) ,
\end{align}
hence
\begin{align}\label{4.23}
u_1  = \sqrt{|D u|^2 - \sum_{i=2}^n u_i^2} \geq \frac{1}{2} |D u| \geq \frac{1}{4} w_1.
\end{align}

Also we have at $x_0$,
\begin{align}\label{4.24}
0 = G_i  = \frac{{(|D w|^2 )_i }}{{|D w|^2 }} + h'u_i  + \alpha _0 d_i,
\end{align}
hence
\begin{align}\label{4.25}
\frac{{2w_{1i} }}{{w_1 }} =  - [h'u_i  + \alpha _0 d_i ].
\end{align}

From the definition of $w$, we know
\begin{align}\label{4.26}
w_{1i}  =& (1 - d)u_{1i}  + [\varphi  - u]d_{1i}  + \varphi _{1i} d  \notag \\
  &+ [\varphi _1  - u_1 ]d_i  + [\varphi _i  - u_i ]d_1.
\end{align}
So we have
\begin{align}\label{4.27}
u_{11}  =& \frac{{w_{11} }}{{1 - d}} - \frac{{[\varphi  - u]d_{11}  + \varphi _{11} d + 2[\varphi _1  - u_1 ]d_1 }}{{1 - d}}  \notag \\
=& \frac{{ - [h'u_1  + \alpha _0 d_1 ]w_1 }}{{2(1 - d)}} - \frac{{[\varphi  - u]d_{11}  + \varphi _{11} d + 2[\varphi _1  - u_1 ]d_1 }}{{1 - d}} \notag \\
\le& \frac{{ - h'}}{{2(1 - d)}}u_1 w_1  + \frac{{\alpha _0 w_1 }}{{2(1 - d)}} + \frac{{( |\varphi | + |u|) |d_{11}| + |\varphi _{11} | + 2|\varphi _1 |}}{{1 - d}} + \frac{{2u_1 }}{{(1 - d)}} \notag \\
\le& \frac{{ - h'}}{{4(1 - d)}}u_1 w_1 \le  - \frac{1}{{16(1 + 2M_0 )}}w_1 ^2  < 0,
\end{align}
where we have assumed $w_1  \ge 8(1 + 2M_0 )[\alpha _0  + 8+ ( |\varphi |_{C^0} + |u|_{C^0}) |D^2 d|_{C^0} + |D^2 \varphi|_{C^0} + 2|D \varphi|_{C^0}]$ ( otherwise there is nothing to prove). Moreover, for $i =1, \cdots, n$, we can get
\begin{align}\label{4.28}
|u_{1i} | =& |\frac{{w_{1i} }}{{1 - d}} - \frac{{[\varphi  - u]d_{1i}  + \varphi _{1i} d + [\varphi _1  - u_1 ]d_i  + [\varphi _i  - u_i ]d_1 }}{{1 - d}}| \notag \\
=& |\frac{{ - [h'u_i  + \alpha _0 d_i ]w_1 }}{{2(1 - d)}} - \frac{{[\varphi  - u]d_{1i}  + \varphi _{1i} d + [\varphi _1  - u_1 ]d_i  + [\varphi _i  - u_i ]d_1 }}{{1 - d}}|  \notag \\
\le& C_4 w_1 ^2.
\end{align}

Denote
\[
F(D^2 u ) = \frac{{\sigma _k (D^2 u )}}{{\sigma _l (D^2 u )}},\quad \text{and } \quad F^{ij}  = \frac{{\partial F}}{{\partial u_{ij} }}.
\]
Then we have
\[
G_{ij}  = \frac{{(|D w|^2 )_{ij} }}{{|D w|^2 }} - \frac{{(|D w|^2 )_i }}{{|D w|^2 }}\frac{{(|D w|^2 )_j }}{{|D w|^2 }} + h'u_{ij}  + h''u_i u_j  + \alpha _0 d_{ij}, \notag
\]
and
\begin{align}\label{4.29}
0 \ge& \sum\limits_{ij = 1}^n {F^{ij} G_{ij} }  = \frac{{F^{ij} (|D w|^2 )_{ij} }}{{|D w|^2 }} - F^{ij} \frac{{(|D w|^2 )_i }}{{|D w|^2 }}\frac{{(|D w|^2 )_j }}{{|D w|^2 }}  \notag \\
&+ F^{ij} [h'u_{ij}  + h''u_i u_j  + \alpha _0 d_{ij} ] \notag \\
=& \frac{{2F^{ij} [\sum\limits_{p = 2}^n {w_{pi} w_{pj} }  + w_{1i} w_{1j}  + w_1 w_{1ij} ]}}{{w_1 ^2 }} - F^{ij} \frac{{2w_{1i} }}{{w_1 }}\frac{{2w_{1j} }}{{w_1 }} \notag \\
&+ F^{ij} [h'u_{ij}  + h''u_i u_j  + \alpha _0 d_{ij} ] \notag \\
\ge& \frac{{2F^{ij} w_{1ij} }}{{w_1 }} - \frac{1}{2}F^{ij} \frac{{2w_{1i} }}{{w_1 }}\frac{{2w_{1j} }}{{w_1 }}+ F^{ij} [h'u_{ij}  + h''u_i u_j  + \alpha _0 d_{ij} ] \notag \\
=& \frac{{2F^{ij} w_{1ij} }}{{w_1 }} - \frac{1}{2}F^{ij} [h'u_i  + \alpha _0 d_i ][h'u_j  + \alpha _0 d_j ] \notag \\
&+ (k - l)h'f + F^{ij} [h''u_i u_j  + \alpha _0 d_{ij} ] \notag  \\
\ge& \frac{{2F^{ij} w_{1ij} }}{{w_1 }} + F^{ij} [(h'' - \frac{1}{2}h'^2 )u_i u_j  - \alpha _0 h'd_i u_j  + \alpha _0 d_{ij}  - \frac{1}{2}\alpha _0 d_i d_j ].
\end{align}
It is easy to know
\begin{align}\label{4.30}
 &F^{ij} [(h'' - \frac{1}{2}h'^2 )u_i u_j  - \alpha _0 h'd_i u_j  + \alpha _0 d_{ij}  - \frac{1}{2}\alpha _0 d_i d_j ]  \notag \\
\ge& \frac{1}{{2(1 + 2M_0 )}}[F^{11} u_1 ^2  -2 \sum\limits_{i=2}^n |F^{1i} u_1 u_i |]  \notag \\
&- \alpha _0 h'|D u|\sum\limits_{i = 1}^n {F^{ii} }  - \alpha _0 ( |D^2 d|  + 1)\sum\limits_{i = 1}^n {F^{ii} }  \notag\\
\ge& \frac{1}{{32(1 + 2M_0 )}}F^{11} w_1 ^2  - C_5 w_1 \sum\limits_{i = 1}^n {F^{ii} }  - C_5\sum\limits_{i = 1}^n {F^{ii} }.
\end{align}
From the definition of $w$, we know
\begin{align}\label{4.31}
w_{ij1}  =& (1 - d)u_{ij1}  + [\varphi  - u]d_{ij1}  + \varphi _{ij1} d \notag \\
 &+ [\varphi _{ij}  - u_{ij} ]d_1  + [\varphi _{i1}  - u_{i1} ]d_j  + [\varphi _{j1}  - u_{j1} ]d_i  \notag \\
 &+ [\varphi _i  - u_i ]d_{1j}  + [\varphi _1  - u_1 ]d_{ij}  + [\varphi _j  - u_j ]d_{i1},
\end{align}
so
\begin{align}\label{4.32}
\frac{{2F^{ij} w_{1ij} }}{{w_1 }} =& \frac{2}{{w_1 }}[(1 - d)F^{ij} u_{ij1}  - d_1 F^{ij} u_{ij}  - 2F^{ij} u_{i1} d_j ]  \notag \\
&+ \frac{2}{{w_1 }}F^{ij} [(\varphi  - u)d_{ij1}  + \varphi _{ij1} d + \varphi _{ij} d_1  + 2\varphi _{i1} d_j  + 2(\varphi _i  - u_i )d_{1j}  + (\varphi _1  - u_1 )d_{ij} ]  \notag  \\
\ge& \frac{2}{{w_1 }}[(1 - d)f_1  - (k - l)d_1 f - 2F^{ij} u_{i1} d_j ] - C\frac{{\sum\limits_{i = 1}^n {F^{ii} } }}{{w_1 }} - C\sum\limits_{i = 1}^n {F^{ii} }  \notag \\
\ge&  - \frac{C_6}{{w_1 }} - C_6w_1 \sum\limits_{i = 1}^n {F^{ii} }  - C_6\frac{{\sum\limits_{i = 1}^n {F^{ii} } }}{{w_1 }} - C_6\sum\limits_{i = 1}^n {F^{ii} }.
\end{align}

From \eqref{4.29}, \eqref{4.30} and \eqref{4.32}, we get
\begin{align}\label{4.33}
0 \ge& \sum\limits_{ij = 1}^n {F^{ij} G_{ij} } \notag \\
\geq& \frac{1}{{32(1 + 2M_0 )}}F^{11} w_1 ^2  - (C_5+C_6)w_1 \sum\limits_{i = 1}^n {F^{ii} } - (C_5+C_6)\sum\limits_{i = 1}^n {F^{ii} } \notag \\
 &- \frac{C_6}{{w_1 }} - C_6\frac{{\sum\limits_{i = 1}^n {F^{ii} } }}{{w_1 }}.
\end{align}

From Lemma \ref{lem2.6}, we know
\begin{align}\label{4.34}
F^{11} \geq c_2 \sum {F^{ii} }
\end{align}
where $c_2 =\frac{{n(k - l)}}{{k(n - l)}}\frac{1}{{n - k + 1}}$. Moreover,
\begin{align}\label{4.35}
\sum {F^{ii} } \geq c_3 (-u_{11})^{k-l-1} \geq c_3  [\frac{1}{{16(1 + 2M_0 )}}w_1 ^2]^{k-l-1}
\end{align}
where $c_3 =\frac{{k - l}}{k} \frac{1}{C_n^l}$. Then we can get from \eqref{4.33}, \eqref{4.34} and \eqref{4.35}
\begin{align}
 w_1 (x_0) \leq& C_7. \notag
\end{align}
So we can prove \eqref{4.4} by a calculation similar with \eqref{4.10}.
\end{proof}

As discussed in Remark \ref{rmk1.4}, we need to consider the equation \eqref{1.4} to prove Theorem \ref{th1.2}. It is crucial to establish a global gradient estimate of $u^\varepsilon$ independent of $\varepsilon$, and we need the strict convexity of $\Omega$. Following the idea of \cite{QX16}, we can easily obtain
\begin{theorem} \label{th4.4}
Suppose $\Omega \subset \mathbb{R}^n$ is a $C^3$ strictly convex domain, $f \in C^1(\overline{\Omega})$ is a positive function, $\varphi \in C^3 (\partial \Omega)$ and $u^\varepsilon \in C^3(\Omega) \cap C^2(\overline \Omega)$ is the $k$-admissible solution of Hessian quotient equation \eqref{1.4} with $\varepsilon >0$ sufficiently small, then we have
\begin{align}\label{4.36}
\sup_{\Omega} |D u^\varepsilon |  \leq \overline{M_1},
\end{align}
and
\begin{align}\label{4.37}
\sup_{\Omega} | u^\varepsilon - \frac{1}{|\Omega|} \int_\Omega u^\varepsilon |  \leq \overline{M_1},
\end{align}
where $\overline{M_1}$ depends on $n$, $k$, $l$, $\Omega$, $|f|_{C^1}$ and $|\varphi|_{C^3}$.
\end{theorem}

\section{Global second derivatives estimate}

We now come to the a priori estimates of global second derivatives, and we obtain the following theorem

\begin{theorem} \label{th5.1}
Suppose that $\Omega \subset \mathbb{R}^n$ is a $C^4$ convex and strictly $(k-1)$-convex domain, $f \in C^2(\overline{\Omega})$ is a positive function, $\varphi \in C^3(\partial \Omega)$ and $u \in C^4(\Omega)\cap C^3(\overline \Omega)$ is the $k$-admissible solution of Hessian quotient equation \eqref{1.2}, then we have
\begin{align}\label{5.1}
\sup_{\Omega} |D^2 u|  \leq M_2,
\end{align}
where $M_2$ depends on $n$, $k$, $l$, $\Omega$, $|u|_{C^1}$, $\inf f$, $|f|_{C^2}$ and $|\varphi|_{C^3}$.
\end{theorem}

Following the idea of Lions-Trudinger-Urbas \cite{LTU86} and Ma-Qiu \cite{MQ15}, we divide the proof of Theorem \ref{th5.1} into three steps. In step one, we reduce global second derivatives to double normal second derivatives on boundary, then we prove the lower estimate of double normal second derivatives on the boundary in step two, and at last we prove the upper estimate of double normal second derivatives on the boundary.

\subsection{Reduce global second derivatives to double normal second derivatives on the boundary}

\begin{lemma} \label{lem5.2}
Suppose that $\Omega \subset \mathbb{R}^n$ is a $C^4$ convex domain, $f \in C^2(\overline{\Omega})$ is a positive function, $\varphi \in C^3(\partial \Omega)$ and $u \in C^4(\Omega)\cap C^3(\overline \Omega)$ is the $k$-admissible solution of Hessian quotient equation \eqref{1.2}, then we have
\begin{align}\label{5.2}
\sup_{\Omega} |D^2 u|  \leq C_9(1+ \max_{\partial \Omega} |u_{\nu \nu}|),
\end{align}
where $C_9$ depends on $n$, $k$, $l$, $\Omega$, $|u|_{C^1}$, $\inf f$, $|f|_{C^2}$ and $|\varphi|_{C^3}$.
\end{lemma}

\begin{proof}
Since $\Omega$ is a $C^4$ domain, it is well known that there exists a small positive universal constant $0<{\mu} < \frac{1}{10}$ such that
$d(x) \in C^4(\overline{\Omega_{\mu}})$ and $\nu = - D d$ on $\partial \Omega$. We define $\widetilde{d} \in C^4(\overline{\Omega})$ such that
$\widetilde{d} =d$ in $\overline{\Omega_{\mu}}$ and denote
\begin{align*}
\nu = - D \widetilde{d}, \quad \text{ in } \Omega.
\end{align*}
In fact, $\nu$ is a $C^3(\overline{\Omega})$ extension of the outer unit normal vector field on $\partial \Omega$.

We assume $0 \in \Omega$, and consider the function
\begin{align}\label{5.3}
v(x, \xi) = u_{\xi \xi} -v' (x, \xi) + K |x|^2 +  |D u|^2,
\end{align}
where $v' (x, \xi) = 2 (\xi \cdot \nu) \xi' (D \varphi - D u- u_l D \nu^l)= a^l u_l +b$, $\xi' = \xi - (\xi \cdot \nu) \nu$, $ a^l = - 2 (\xi \cdot \nu) (\xi' \cdot D \nu^l)- 2 (\xi \cdot \nu) (\xi')^l$, $b = 2 (\xi \cdot \nu) (\xi' \cdot D \varphi)$, and
\begin{align*}
K =& [\frac{{C_n^l }} {{C_n^k }}]^{\frac{1} {{k - l}}} \Big[|D^2 (f^{\frac{1} {{k - l}}} )|_{C^0}+  |a^l|_{C^0} |D (f^{\frac{1} {{k - l}}} )|_{C^0} + 2|D (f^{\frac{1} {{k - l}}} )|_{C^0} | D u|_{C^0} \Big]  \\
& + {|D a^l|_{C^0}}^2 + |D^2 a^l|_{C^0} |D u|_{C^0} + |D^2 b|_{C^0}.
\end{align*}

Denote
\[
\widetilde{F}(D^2 u ) = [\frac{{\sigma _k (D^2 u )}}{{\sigma _l (D^2 u )}}]^{\frac{1}{k-l}},\quad \text{and } \quad \widetilde{F}^{ij}  = \frac{{\partial \widetilde{F}}}{{\partial u_{ij} }}.
\]

For any fixed $\xi \in \mathbb{S}^{n-1}$, we have
\begin{align}
\widetilde{ F}^{ij} v_{ij}  = &\sum\limits_{i j= 1}^n {\widetilde{ F}^{ij} u_{ij\xi \xi }}-  \sum\limits_{i j= 1}^n {\widetilde{F}^{ij} [a^l u_{ijl} + 2 D_i a^l u_{jl} + D_{ij} a^l u_l +b_{ij}]} \notag \\
 &+ 2K\sum\limits_{i = 1}^n {\widetilde{F}^{ii} }  + \sum\limits_{i j= 1}^n {\widetilde{F}^{ij} [2u_{ijk}u_k + 2u_{ki} u_{kj}]} \notag \\
=& (f^{\frac{1} {{k - l}}} )_{\xi \xi }  - \widetilde{F}^{ij,kl} u_{ij\xi \xi }  - a^l (f^{\frac{1} {{k - l}}} )_{l}- \sum\limits_{i j= 1}^n {\widetilde{F}^{ij} [ 2 D_i a^l u_{jl} + D_{ij} a^l u_l +b_{ij}]} \notag \\
&+ 2K\sum\limits_{i = 1}^n {\widetilde{F}^{ii} }  + 2(f^{\frac{1} {{k - l}}} )_{k} u_k +  2\sum\limits_{i j= 1}^n {\widetilde{F}^{ij} u_{ki} u_{kj}} \notag \\
\geq& (f^{\frac{1} {{k - l}}} )_{\xi \xi }  - a^l (f^{\frac{1} {{k - l}}} )_{l}- \sum\limits_{i j= 1}^n {\widetilde{F}^{ij} [  D_i a^l D_j a^l + D_{ij} a^l u_l +b_{ij}]} \notag \\
 &+ 2K\sum\limits_{i = 1}^n {\widetilde{F}^{ii} }  + 2(f^{\frac{1} {{k - l}}} )_{k} u_k \notag \\
\geq& (f^{\frac{1} {{k - l}}} )_{\xi \xi }  - a^l (f^{\frac{1} {{k - l}}} )_{l}- 2|D (f^{\frac{1} {{k - l}}} )| | D u| \notag \\
&+ [2K -  |D a^l|^2 - |D^2 a^l| |D u| - |D^2 b|] \sum\limits_{i = 1}^n {\widetilde{F}^{ii} } \notag \\
\geq& (f^{\frac{1} {{k - l}}} )_{\xi \xi }  - a^l (f^{\frac{1} {{k - l}}} )_{l}- 2|D (f^{\frac{1} {{k - l}}} )| | D u| \notag \\
&+ [2K -  |D a^l|^2 - |D^2 a^l| |D u| - |D^2 b|] \cdot [\frac{{C_n^k }} {{C_n^l }}]^{\frac{1} {{k - l}}}  \notag \\
 >& 0, \notag
\end{align}
where we have used
\begin{align}\label{5.4}
 2\sum\limits_{i j= 1}^n {\widetilde{F}^{ij} D_i a^l u_{jl}} \leq \sum\limits_{i j= 1}^n {\widetilde{F}^{ij} u_{ki} u_{kj}} +\sum\limits_{i j= 1}^n {\widetilde{F}^{ij}  D_i a^l D_j a^l}.
\end{align}
So $\max\limits_\Omega v(x, \xi)$ attains at a point on $\partial \Omega$. Hence $\max\limits_{\Omega \times \mathbb{S}^{n-1}} v(x, \xi)$ attains at some point $x_0 \in \partial \Omega$ and some direction $\xi_0 \in \mathbb{S}^{n-1}$.

Case a: $\xi_0$ is tangential to $\partial \Omega$ at $x_0$.

We directly have $\xi_0 \cdot \nu =0$, $v'(x_0, \xi_0) =0$, and $u_{\xi_0 \xi_0} (x_0)>0$. In the following, the calculations are at the point $x_0$ and $\xi = \xi_0$.

From the boundary condition, we have
\begin{align}\label{5.5}
u_{li}\nu^l =& [c^{ij} + \nu^i \nu^j ]\nu^l u_{lj}\notag  \\
=& c^{ij} [D_j(\nu^l u_{l}) -D_j \nu^l u_l ] + \nu^i \nu^j \nu^l u_{lj} \notag \\
=& - c^{ij} u_j + c^{ij} D_j \varphi - c^{ij}u_l D_j \nu^l + \nu^i \nu^j \nu^l u_{lj}.
\end{align}
So it follows that
\begin{align}\label{5.6}
u_{lip} \nu^l =& [c^{pq} + \nu^p \nu^q ] u_{liq} \nu^l \notag \\
=& c^{pq} [D_q (u_{li} \nu^l)- u_{li} D_q \nu^l] + \nu^p \nu^q u_{liq} \nu^l \notag \\
=& c^{pq} D_q (- c^{ij} u_j + c^{ij} D_j \varphi - c^{ij}u_lD_j \nu^l + \nu^i \nu^j \nu^l u_{lj}) \notag \\
&-c^{pq} u_{li} D_q \nu^l + \nu^p \nu^q \nu^l u_{liq},
\end{align}
then we obtain
\begin{align}\label{5.7}
u_{\xi_0 \xi_0 \nu} =& \sum_{ip l=1}^n  \xi_0^i \xi_0^p u_{lip} \nu^l \notag \\
=& \sum_{ip=1}^n  \xi_0^i \xi_0^p[c^{pq}D_q(- c^{ij} u_j + c^{ij} D_j \varphi-c^{ij}u_lD_j \nu^l + \nu^i \nu^j \nu^l u_{lj}) \notag \\
&\qquad \qquad - c^{pq}u_{li} D_q \nu^l + \nu^p \nu^q \nu^l u_{liq} ]\notag \\
=& \sum_{i=1}^n  \xi_0^i \xi_0^q[D_q(- c^{ij} u_j +c^{ij} D_j \varphi-c^{ij}u_lD_j \nu^l + \nu^i \nu^j \nu^l u_{lj}) - u_{li} D_q \nu^l ]\notag \\
=& -\xi_0^i \xi_0^q  [c^{ij} u_{jq} -D_q c^{ij} u_{j}] +  \xi_0^i \xi_0^q D_q(c^{ij} D_j \varphi ) \notag \\
&- \xi_0^i \xi_0^q D_q( c^{ij}D_j \nu^l ) u_l - \xi_0^j \xi_0^q u_{lq} D_j \nu^l + \xi_0^i \xi_0^q D_q \nu^i u_{\nu \nu}  - \xi_0^i \xi_0^q u_{lq} D_i \nu^l\notag \\
\leq& -u_{\xi_0 \xi_0} - 2 \xi_0^i u_{l \xi_0} D_i \nu^l + C_{10} + C_{10}|D u| + C_{10}| u_{\nu \nu}|.
\end{align}
 We assume $\xi_0 = e_1$, it is easy to get the bound for $u_{1i}(x_0)$ for $i > 1$ from the maximum
of $v(x, \xi)$ in the $\xi_0$ direction. In fact, we can assume $\xi(t) = \frac{(1,t,0, \cdots,0)}{\sqrt {1+t^2}}$.
Then we have
\begin{align}\label{5.8}
  0 =& \frac{{dv(x_0 ,\xi (t))}} {{dt}}|_{t = 0}  \notag \\
   =& 2u_{ij} (x_0 )\frac{{d\xi ^i (t)}} {{dt}}|_{t = 0} \xi ^j (0) - \frac{{dv'(x_0 ,\xi (t))}}
{{dt}}|_{t = 0}  \notag \\
   =& 2u_{12} (x_0 ) - 2\nu ^2 (D _1 \varphi -u_1 - u_l D _1 \nu ^l ),
\end{align}
so
\begin{align}\label{5.9}
|u_{12} (x_0 )|= |\nu ^2 (D _1 \varphi -u_1 - u_l D _1 \nu ^l )| \leq C_{11}+ C_{11}|D u|.
\end{align}
Similarly, we have for all $i > 1$,
\begin{align}\label{5.10}
|u_{1i} (x_0 )|\leq  C_{11}+ C_{11}|D u|.
\end{align}
so by $ \{D_i \nu^l\} \geq 0$, we have
\begin{align}\label{5.11}
u_{\xi_0 \xi_0 \nu} \leq& - u_{\xi_0 \xi_0} - D_1 \nu^1 u_{\xi_0 \xi_0}  + C_{12} (1+ | u_{\nu \nu}|) \notag \\
\leq& - u_{\xi_0 \xi_0} + C_{12} (1+ | u_{\nu \nu}|).
\end{align}
On the other hand, we have from the Hopf lemma,  and \eqref{5.5},
\begin{align}\label{5.12}
0 \leq& v_\nu (x_0, \xi_0)\notag \\
=& u_{\xi_0 \xi_0 \nu} -  a^l u_{l\nu}  -  D_\nu a^l u_{l}- b_\nu + 2 K (x \cdot \nu) \notag \\
\leq& -u_{\xi_0 \xi_0} + C_{12} (1+ | u_{\nu \nu}|) + C_{13}.
\end{align}
Then we get
\begin{align}\label{5.13}
u_{\xi_0 \xi_0}(x_0) \leq  (C_{12} + C_{13}) (1+ | u_{\nu \nu}|),
\end{align}
and
\begin{align}\label{5.14}
 \max\limits_{\Omega \times \mathbb{S}^{n-1}} |u_{\xi \xi}(x)| \leq& (n-1) \max\limits_{\Omega \times \mathbb{S}^{n-1}} u_{\xi \xi}(x) \notag \\
 \leq& (n-1)[ \max\limits_{\Omega \times \mathbb{S}^{n-1}} v(x, \xi)  + C_{14}] = (n-1) [v(x_0, \xi_0) + C_{14} ]\notag \\
 \leq& (n-1) [u_{\xi_0 \xi_0}(x_0) +2 C_{14}] \notag \\
\leq& C_{15} (1+ | u_{\nu \nu}|).
\end{align}

Case b: $\xi_0$ is non-tangential.

We can directly have $\xi_0 \cdot \nu \ne 0$. We can find a tangential vector $\tau$, such that $\xi_0 = \alpha \tau + \beta \nu$, with $\alpha = \xi_0 \cdot \tau \geq 0$, $\beta = \xi_0 \cdot \nu \ne 0$, $\alpha ^2 + \beta ^2 =1$ and $\tau \cdot \nu =0$. Then we have

\begin{align}\label{5.15}
u_{\xi_0 \xi_0}(x_0) =& \alpha^2 u_{\tau \tau}(x_0) + \beta^2u_{\nu \nu}(x_0)+ 2 \alpha \beta u_{\tau \nu}(x_0) \notag \\
=& \alpha^2 u_{\tau \tau}(x_0) + \beta^2u_{\nu \nu}(x_0)+ 2  (\xi_0 \cdot \nu)[\xi_0 - (\xi_0 \cdot \nu) \nu] [D \varphi - D u- u_l D \nu^l],
\end{align}
hence
\begin{align}\label{5.16}
v(x_0, \xi_0) = \alpha^2 v(x_0, \tau) + \beta^2 v(x_0, \nu).
\end{align}
From the definition of $v(x_0, \xi_0)$, we know
\begin{align}\label{5.17}
v(x_0, \xi_0)= v(x_0, \nu),
\end{align}
and
\begin{align}\label{5.18}
u_{\xi_0 \xi_0}(x_0) \leq v(x_0, \xi_0) + C_{14} = v(x_0, \nu) +C_{14}  \leq | u_{\nu \nu}|+ 2C_{14}.
\end{align}
Similarly as \eqref{5.14}, we can prove \eqref{5.2}.
\end{proof}

\subsection{Lower estimate of double normal second derivatives on boundary}

\begin{lemma} \label{lem5.3}
Suppose that $\Omega \subset \mathbb{R}^n$ is a $C^3$ convex and strictly $(k-1)$-convex domain, $f \in C^2(\overline{\Omega})$ is a positive function, $\varphi \in C^3(\partial \Omega)$ and $u \in C^3(\Omega)\cap C^2(\overline \Omega)$ is the $k$-admissible solution of Hessian quotient equation \eqref{1.2}, then we have
\begin{align}\label{5.19}
\min_{\partial \Omega} u_{\nu \nu}  \geq - C_{15},
\end{align}
where $C_{15}$ is a positive constants depending on $n$, $k$,$l$, $\Omega$, $|u|_{C^1}$, $\inf f$, $|f|_{C^2}$ and $|\varphi|_{C^3}$.
\end{lemma}

To prove Lemma \ref{lem5.3} and Lemma \ref{lem5.5}, we need the following lemma.
\begin{lemma} \label{lem5.4}
Suppose $\Omega \subset \mathbb{R}^n$ is a $C^2$ convex and strictly $(k-1)$-convex domain, $f \in C^2(\overline{\Omega})$ is a positive function, and $u \in C^2(\Omega)$ is the $k$-admissible solution of Hessian quotient equation
\begin{align}
\frac{\sigma _k (D^2 u)}{\sigma _l (D^2 u)} = f(x),  \quad \text{in} \quad \Omega \subset \mathbb{R}^n. \notag
\end{align}
Denote $ F^{ij}  = \frac{{\partial \frac{{\sigma _k (D^2 u )}}{{\sigma _l (D^2 u )}}}}{{\partial u_{ij} }}$, and
\begin{align}\label{5.20}
h(x) = -d(x) + d^2(x),
\end{align}
where $d(x) =dist(x, \partial \Omega)$ is the distance function of $\Omega$. Then
\begin{align}\label{5.21}
 \sum\limits_{ij=1}^n F^{ij} h_{ij} \geq c_4 (\sum\limits_{i=1}^n F^{ii} + 1), \quad  \text{ in } \Omega_\mu,
\end{align}
where $\Omega_\mu = \{ x \in \Omega: d(x) < \mu\}$ for a small universal constant $\mu$ and $c_4$ is a positive constant depending only on $n$, $k$, $l$, $\Omega$ and $\inf f$.
\end{lemma}

\begin{proof}
We know from the classic book \cite{GT} section 14.6 that the distance function $d$ is $C^4$ in
$\Omega_\mu = \{ x \in \Omega: 0 < d(x) < \mu\}$ for some constant $\mu \in (0, \frac{1}{10})$ small depending on $\Omega$. Also it holds
\begin{align}\label{5.22}
|D d| =1, \quad \text{ in } \overline{\Omega_\mu}; \quad  -D d = \nu, \quad \text{ on } \partial \Omega_\mu.
\end{align}
For any $x_0 \in \Omega_\mu$, there is a $y_0 \in \partial \Omega$ such that $d(x_0) = |x_0 - y_0|$. In the principal coordinate system (see \cite{GT} section 14.6), we have
\begin{align}
\label{5.23}&-D d(x_0) = \nu (y_0) =(0, \cdots, 0, 1); \\
\label{5.24}&-D^2 d(x_0) = \textrm{diag} \{ \frac{\kappa_1(y_0)}{1- \kappa_1(y_0)d(x_0)}, \cdots, \frac{\kappa_{n-1}(y_0)}{1- \kappa_{n-1}(y_0)d(x_0)}, 0  \},
\end{align}
where $\kappa_1(y_0), \cdots, \kappa_{n-1}(y_0)$ are the principal curvature of $\partial \Omega$ at $y_0$. Since $\Omega$ is convex and strictly $(k-1)$-convex, then there exist two positive constants $\kappa_{min}<1$ and $\kappa_{max}$ depending only on $\Omega$ and $\mu$ such that
\begin{align}\label{5.25}
\kappa_{min} \textrm{diag}\{ \underbrace {1, \cdots, 1,}_{k - 1} 0, \cdots, 0\} \leq -D^2 d(x_0) \leq  \kappa_{max} \textrm{diag} \{ 1, \cdots, 1, 0 \},
\end{align}
in the principal coordinate system. Hence
\begin{align}\label{5.26}
\kappa_{min} \textrm{diag}\{ \underbrace {1, \cdots, 1,}_{k - 1} 0, \cdots, 0, 1\} \leq D^2 h(x_0) \leq  (\kappa_{max}+1) \textrm{diag} \{ 1, \cdots, 1, 1 \},
\end{align}
in the principal coordinate system.

If $D^2 u(x_0)$ is diagonal, and denote $\lambda = (\lambda_1, \cdots, \lambda_n)$ with $\lambda_i = u_{ii}$. We also assume $\lambda_1 \geq \lambda_2 \geq  \cdots \geq \lambda_n$.  We can easily get
\begin{align}
F^{ii} =& \frac{{\sigma _{k-1} (\lambda|i)\sigma _{l} (\lambda) - \sigma _{k} (\lambda |i)\sigma _{l-1} (\lambda)  }}{{\sigma _l^2 (\lambda)}} \notag \\
\geq& (1 - \frac{l} {k}\frac{{n - k}} {{n - l}})\frac{{\sigma _{k-1} (\lambda |i)\sigma _{l} (\lambda |i) }}{{\sigma _l^2 (\lambda )}}. \notag
\end{align}
Then
\begin{align}
 \sum\limits_{i,j=1}^n F^{ij} h_{ij} \geq& (1 - \frac{l} {k}\frac{{n - k}} {{n - l}}) \sum\limits_{i=1}^n \frac{{\sigma _{k-1} (\lambda |i)\sigma _{l} (\lambda |i) }}{{\sigma _l^2 (\lambda)}}h_{ii}  \notag \\
\geq& \kappa_{min} \frac{{\sigma _{k-1} (\lambda |k)\sigma _{l} (\lambda |k) }}{{\sigma _l^2 (\lambda)}}, \notag
\end{align}
From Lin-Trudinger \cite{LT94}, we know $\sigma_{k-1}(\lambda |k) \geq c(n,k) \sigma_{k-1}(\lambda)$ for some positive constant $c(n,k)$ depending only on $n$ and $k$.
So
\begin{align}\label{5.27}
 \sum\limits_{i,j=1}^n F^{ij} h_{ij} \geq&\kappa_{min} \frac{{\sigma _{k-1} (\lambda |k)\sigma _{l} (\lambda |k) }}{{\sigma _l^2 (\lambda)}} \notag \\
\geq& \kappa_{min} c(n,k) c(n,l) \frac{{\sigma _{k-1} (\lambda)\sigma _{l} (\lambda ) }}{{\sigma _l^2 (\lambda)}} \notag \\
\geq& c \sum\limits_{i=1}^n F^{ii}.
\end{align}
Also, we have
\begin{align}\label{5.28}
\sum\limits_{i=1}^n F^{ii} =& \frac{{(n-k+1)\sigma _{k-1} (D^2 u )\sigma _{l} (D^2 u ) - (n-l+1) \sigma _{k} (D^2 u )\sigma _{l-1} (D^2 u )  }}{{\sigma _l^2 (D^2 u )}} \notag \\
\geq&  \frac{k-l}{k}(n-k+1)\frac{{\sigma _{k-1} (D^2 u )}}{{\sigma _l (D^2 u )}}
 \notag \\
\geq&  c(n,k,l) [\frac{{\sigma _{k} (D^2 u )}}{{\sigma _l (D^2 u )}}]^{\frac{k-l-1}{k-l}} = c(n,k,l) f ^{\frac{k-l-1}{k-l}} \geq c_5 >0.
\end{align}
Hence from \eqref{5.27} and \eqref{5.28}, we get
\begin{align}
 \sum\limits_{ij=1}^n F^{ij} h_{ij} \geq& c \sum\limits_{i=1}^n F^{ii} \geq \frac{c}{2} \sum\limits_{i=1}^n F^{ii} + \frac{c}{2} c_5 \geq c_4 (\sum\limits_{i=1}^n F^{ii} + 1), \quad  \text{ in } \Omega_\mu. \notag
\end{align}
\end{proof}

Now we come to prove Lemma \ref{lem5.3}.
\begin{proof}

Firstly, we assume $ \min\limits_{\partial \Omega} u_{\nu \nu} < 0$, otherwise there is nothing to prove. Also, if $- \min\limits_{\partial \Omega} u_{\nu \nu} < \max\limits_{\partial \Omega} u_{\nu \nu}$, that is $\max\limits_{\partial \Omega} |u_{\nu \nu}|= \max\limits_{\partial \Omega} u_{\nu \nu}$, we can easily get from Lemma \ref{lem5.4}
\begin{align}
- \min_{\partial \Omega} u_{\nu \nu}  < \max_{\partial \Omega} u_{\nu \nu} \leq C_{18}.  \notag
\end{align}
In the following, we assume $- \min\limits_{\partial \Omega} u_{\nu \nu} \geq \max\limits_{\partial \Omega} u_{\nu \nu}$, that is $\max\limits_{\partial \Omega} |u_{\nu \nu}|= -\min\limits_{\partial \Omega} u_{\nu \nu}$. Denote $M = - \min\limits_{\partial \Omega} u_{\nu \nu} >0$ and let $z_0 \in \partial \Omega$ such that $\min\limits_{\partial \Omega} u_{\nu \nu} =u_{\nu \nu} (z_0)$.

Motivated by Ma-Qiu \cite{MQ15}, we consider the test function
\begin{align}\label{5.29}
P(x) = ( 1 + \beta d)[D u \cdot (-D d) + u(x)  - \varphi (x) ] + (A + \frac{1} {2}M)h(x),
\end{align}
where
\begin{align}\label{5.30}
\beta =&\max \{\frac{1}{\mu}, 5 n (2\kappa_{max}+\frac{1}{n}) \frac{C_9}{c_6} \}, \\
A=& \max\{A_1, A_2, \frac{C_{17} +\frac{2 (k-l)}{n} f }{c_4} \}.
\end{align}

It is easy to know that $P \leq 0$ on $\partial \Omega_\mu$. Precisely, on $\partial \Omega$, we have $d = h= 0$, and $- D d = \nu$, so we can get
\begin{align}\label{5.32}
P(x) =0,  \quad \text{ on } \partial \Omega.
\end{align}
On $\partial \Omega_\mu \setminus \partial \Omega$, we have $d = \mu$, and
\begin{align}\label{5.33}
P(x) \leq& (1+ \beta \mu) [|D u|  +  |u| + |\varphi|]  + (A + \frac{1} {2}M) [-\mu + \mu^2] \notag \\
\leq& (1+ \beta \mu) [|D u|  +  |u| + |\varphi|]  - \frac{9}{10} \mu A < 0,
\end{align}
since $A \geq \frac{10}{9}(\frac{1}{\mu} + \beta) [|D u|_{C^0}  +  |u|_{C^0} + |\varphi|_{C^0}] +1 =:A_1$. In the following, we want to prove $P$ attains its maximum only on $\partial \Omega$. Then we can get
\begin{align}\label{5.34}
0 \leq P_\nu (z_0)  =&  [u_{\nu \nu}(z_0) - \sum\limits_j {u_j d_{j\nu} }   + u_\nu  - \varphi_\nu] + (A + \frac{1} {2}M)  \notag\\
\leq& \min_{\partial \Omega} u_{\nu \nu} + |D u| |D^2 d| +|D u|  + |D \varphi| + A + \frac{1} {2}M,
\end{align}
hence \eqref{5.19} holds.

To prove $P$ attains its maximum only on $\partial \Omega$, we assume $P$ attains its maximum at some point  $x_0 \in \Omega_\mu$ by contradiction. Rotating the coordinates, we can assume
\begin{align}\label{5.35}
D^2 u(x_0) \text{ is diagonal}.
\end{align}
In the following, all the calculations are at $x_0$.

Firstly, we have
\begin{align}\label{5.36}
0 = P_i =& \beta d_i [-\sum\limits_j {u_j d_j }  + u- \varphi ] + (1+ \beta d)[-\sum\limits_j {(u_{ji} d_j + u_j d_{ji})}   + u_i- \varphi_i ] \notag \\
&+ (A + \frac{1} {2}M)h_i\notag \\
 =&\beta d_i [-\sum\limits_j {u_j d_j }  + u- \varphi ] + (1+ \beta d)[-u_{ii} d_i-\sum\limits_j { u_j d_{ji}}   + u_i- \varphi_i ] \notag \\
&+ (A + \frac{1} {2}M)h_i,
\end{align}
and
\begin{align}\label{5.37}
0 \geq P_{ii}  =& \beta d_{ii} [-\sum\limits_j {u_j d_j } + u- \varphi] + 2\beta d_i [-\sum\limits_j {(u_{ji} d_j + u_j d_{ji})}   + u_i- \varphi_i]  \notag \\
&+ (1+ \beta d)[-\sum\limits_j {(u_{jii} d_j +2 u_{ji} d_{ji} + u_j d_{jii})} + u_{ii} - \varphi_{ii} ] + (A + \frac{1} {2}M)h_{ii}   \notag \\
=&  \beta d_{ii} [-\sum\limits_j {u_j d_j } + u- \varphi] + 2\beta d_i [-u_{ii} d_i-\sum\limits_j {u_j d_{ji}}   + u_i- \varphi_i]  \notag \\
&+ (1+ \beta d)[-\sum\limits_j {u_{jii} d_j }  - 2u_{ii} d_{ii}  - \sum\limits_j {u_j d_{jii} }  + u_{ii} - \varphi_{ii}]  \notag \\
&+ (A + \frac{1} {2}M)h_{ii}   \notag \\
\geq&  - 2\beta u_{ii} d_i ^2  + (1+ \beta d)[-\sum\limits_j {u_{jii} d_j }  - 2u_{ii} d_{ii} + u_{ii}] \notag \\
&+ (A + \frac{1} {2}M)h_{ii} - C_{16},
\end{align}
where $C_{16}$ is a positive constant under control as follows
\begin{align}\label{5.38}
C_{16} =&\beta |D^2 d| \Big[|D u|_{C^0} + |u|_{C^0} + |\varphi|_{C^0}] +2\beta [|D u|_{C^0} |D^2 d|_{C^0}  +|D u|_{C^0}+ |D \varphi|_{C^0}\Big] \notag \\
&+ (1+ \beta \mu)\Big[ |D u|_{C^0} |D^3 d|_{C^0}+ |D^2 \varphi|_{C^0}\Big].
\end{align}

Since $D^2 u(x_0)$ is diagonal, we know $F^{ij} =0$ for $i \ne j$. From the equation \eqref{1.2}, we have
\begin{align}
 \sum\limits_{i = 1}^n {F^{ii} u_{ii} }  =& (k-l) f >0, \notag \\
\sum\limits_{i = 1}^n {F^{ii} u_{iij} }  =&  f_j, \notag
\end{align}
hence
\begin{align}\label{5.39}
0 \geq& \sum\limits_{i = 1}^n {F^{ii} P_{ii} } \notag \\
\geq&  - 2\beta \sum\limits_{i = 1}^n {F^{ii} u_{ii} d_i ^2 }  + (1+ \beta d)[-\sum\limits_{i,j} {F^{ii} u_{jii} d_j } - 2\sum\limits_{i = 1}^n {F^{ii} u_{ii} d_{ii} } +\sum\limits_{i = 1}^n {F^{ii} u_{ii}}]  \notag \\
&+ (A + \frac{1} {2}M)\sum\limits_{i = 1}^n {F^{ii} h_{ii} }  - C_{16}\sum\limits_{i = 1}^n {F^{ii} }   \notag \\
\geq&  - 2\beta \sum\limits_{i = 1}^n {F^{ii} u_{ii} d_i ^2 }  - 2(1+ \beta d)\sum\limits_{i = 1}^n {F^{ii} u_{ii} d_{ii} }   \notag \\
&+ [(A + \frac{1} {2}M)c_4  - C_{17}](\sum\limits_{i = 1}^n {F^{ii} }  + 1),
\end{align}
where $C_{17} = \max \{C_{16}, (1+ \beta \mu)|D f|_{C^0} \}$.

Denote $B = \{ i:\beta d_i ^2  < \frac{1}{n},1 \leq i \leq n\}$ and $G = \{ i:\beta d_i ^2  \geq  \frac{1}{n},1 \leq i \leq n\}$. We choose $\beta \geq \frac{1}{\mu} > 1$, so
\begin{align}\label{5.40}
d_i ^2 < \frac{1}{n} = \frac{1}{n} |D d|^2, \quad i \in B.
\end{align}
It holds $ \sum_{i \in B} d_i ^2 < 1 = |D d|^2$, and $G$ is not empty. Hence for any $i \in G$, it holds
\begin{align}\label{5.41}
d_{i} ^2 \geq \frac{1}{n\beta}.
\end{align}
and from \eqref{5.36}, we have
\begin{align}\label{5.42}
u_{i i }  =  - \frac{{1-2d}}{1+ \beta d}(A + \frac{1}{2}M) + \frac{{\beta [-\sum\limits_j {u_j d_j }  + u- \varphi ]}}
{1+ \beta d} + \frac{{-\sum\limits_j { u_j d_{ji}}   + u_i- \varphi_i }}{{d_{i } }}.
\end{align}
We choose $A \geq 5 \beta \Big[| D u|_{C^0} + |u|_{C^0}+ |\varphi|_{C^0}] + 5 \sqrt {n\beta} [|D u|_{C^0} |D^2 d|_{C^0} + |D u|_{C^0} + |D \varphi|_{C^0}\Big]=:A_2$, such that for any $i \in G$
\begin{align}\label{5.43}
&\Big| \frac{{\beta [-\sum\limits_j {u_j d_j }  + u- \varphi ]}}{1+ \beta d} + \frac{{-\sum\limits_j { u_j d_{ji}}   + u_i- \varphi_i }}{{d_{i } }}\Big|  \notag \\
\leq& \beta [| D u| + |u|+ |\varphi|] + \sqrt {n\beta} [|D u| |D^2 d| + |D u| + |D \varphi|]  \notag \\
\leq& \frac{A}{5},
\end{align}
then we can get
\begin{align}\label{5.44}
 - \frac{{6A}}{5} - \frac{M}{2} \leqslant u_{i i }  \leq  - \frac{A+M}{5}, \quad \forall \quad i \in G.
\end{align}
Also there is an $i_0  \in G$ such that
\begin{align}\label{5.45}
d_{i_0} ^2 \geq \frac{1}{n} |D d|^2 = \frac{1}{n}.
\end{align}
From \eqref{5.39}, we have
\begin{align}\label{5.46}
0 \geq \sum\limits_{i = 1}^n {F^{ii} P_{ii} }  \geq&  - 2\beta \sum\limits_{i \in G} {F^{ii} u_{ii} d_i ^2 }  - 2\beta \sum\limits_{i \in B} {F^{ii} u_{ii} d_i ^2 }  \notag \\
& - 2(1+ \beta d)\sum\limits_{u_{ii}  > 0} {F^{ii} u_{ii} d_{ii} }  - 2(1+ \beta d)\sum\limits_{u_{ii}  < 0} {F^{ii} u_{ii} d_{ii} }  \notag \\
&+ [(A + \frac{1} {2}M)c_4  - C_{17}](\sum\limits_{i = 1}^n {F^{ii} }  + 1) \notag \\
\geq&  - 2\beta \sum\limits_{i \in G} {F^{ii} u_{ii} d_i ^2 }  - 2\beta \sum\limits_{i \in B} {F^{ii} u_{ii} d_i ^2 }  +  4 \kappa_{max} \sum\limits_{u_{ii}  < 0} {F^{ii} u_{ii} }\notag \\
&+ [(A + \frac{1} {2}M)c_4  - C_{17}](\sum\limits_{i = 1}^n {F^{ii} }  + 1),
\end{align}
where $\kappa_{max}$ is defined as in \eqref{5.25}. Direct calculations yield
\begin{align}\label{5.47}
 - 2\beta \sum\limits_{i \in G} {F^{ii} u_{ii} d_i ^2 }  \geq  - 2\beta F^{i_0 i_0 } u_{i_0 i_0 } d_{i_0 } ^2  \geq  - \frac{2\beta }
{{n}}F^{i_0 i_0 } u_{i_0 i_0 },
\end{align}
and
\begin{align}\label{5.48}
- 2\beta \sum\limits_{i \in B} {F^{ii} u_{ii} d_i ^2 }  \geq&  - 2\beta \sum\limits_{i \in B, u_{ii}  > 0} {F^{ii} u_{ii} d_i ^2 }
\geq - \frac{2}{n}  \sum\limits_{i \in B,u_{ii}  > 0} {F^{ii} u_{ii} }  \notag \\
\geq&  - \frac{2}{n}  \sum\limits_{u_{ii}  > 0} {F^{ii} u_{ii} } =  - \frac{2}{n}  [(k-l) f - \sum\limits_{u_{ii}  < 0} {F^{ii} u_{ii} }].
\end{align}

For $u_{i_0 i_0 }<0$, we know from Lemma \ref{lem2.5},
\begin{align}
F^{i_0 i_0 }  \geq c_6 \sum\limits_{i = 1}^n {F^{ii} },
\end{align}
where $c_6 =\frac{{n(k - l)}}{{k(n - l)}}\frac{1}{{n - k + 1}}$. So it holds
\begin{align}\label{5.50}
0 \geq \sum\limits_{i = 1}^n {F^{ii} P_{ii} }  \geq& - \frac{2\beta }
{{n}}F^{i_0 i_0 } u_{i_0 i_0 }  +  (4 \kappa_{max}+\frac{2}{n}) \sum\limits_{u_{ii}  < 0} {F^{ii} u_{ii} }  \notag \\
&+ [(A + \frac{1} {2}M)c_4  - C_{17} - \frac{2 (k-l)}{n} f ](\sum\limits_{i = 1}^n {F^{ii} }  + 1)  \notag \\
\geq& \frac{2\beta } {{n}}c_6  \frac{A+M}{5} \sum\limits_{i = 1}^n {F^{ii} } -  (4\kappa_{max}+\frac{2}{n}) C_9 (1 + M) \sum\limits_{i = 1}^n {F^{ii} }\notag \\
>& 0,
\end{align}
since $\beta \geq 5 n (2\kappa_{max}+\frac{1}{n}) \frac{C_9}{c_6}$. This is a contradiction. So $P$ attains its maximum only on $\partial \Omega$. The proof of Lemma \ref{lem5.3} is complete.

\end{proof}

\subsection{Upper estimate of double normal second derivatives on boundary}

\begin{lemma} \label{lem5.5}
Suppose that $\Omega \subset \mathbb{R}^n$ is a $C^3$ convex and strictly $(k-1)$-convex domain, $f \in C^2(\overline{\Omega})$ is a positive function, $\varphi \in C^3(\partial \Omega)$ and $u \in C^3(\Omega)\cap C^2(\overline \Omega)$ is the $k$-admissible solution of Hessian quotient equation \eqref{1.2}, then we have
\begin{align}\label{5.51}
\max_{\partial \Omega} u_{\nu \nu}  \leq C_{18},
\end{align}
where $C$ depends on $n$, $k$, $l$, $\Omega$, $\inf f$, $|f|_{C^2}$ and $|\varphi|_{C^3}$.
\end{lemma}

\begin{proof}

Firstly, we assume $ \max\limits_{\partial \Omega} u_{\nu \nu} > 0$, otherwise there is nothing to prove. Also, if $\max\limits_{\partial \Omega} u_{\nu \nu} < - \min\limits_{\partial \Omega} u_{\nu \nu}$, that is $\max\limits_{\partial \Omega} |u_{\nu \nu}|= -\min\limits_{\partial \Omega} u_{\nu \nu}$, we can easily get from Lemma \ref{lem5.3}
\begin{align}\label{5.52}
\max_{\partial \Omega} u_{\nu \nu} < - \min_{\partial \Omega} u_{\nu \nu} \leq C_{15}.
\end{align}
In the following, we assume $\max\limits_{\partial \Omega} u_{\nu \nu} \geq - \min\limits_{\partial \Omega} u_{\nu \nu}$, that is $\max\limits_{\partial \Omega} |u_{\nu \nu}|= \max\limits_{\partial \Omega} u_{\nu \nu}$. Denote $M = \max\limits_{\partial \Omega} u_{\nu \nu} >0$ and let $\widetilde{z_0} \in \partial \Omega$ such that $\max\limits_{\partial \Omega} u_{\nu \nu} =u_{\nu \nu} (\widetilde{z_0})$.

Motivated by Ma-Qiu \cite{MQ15}, we consider the test function
\begin{align}\label{5.53}
\widetilde{P}(x) = ( 1 + \beta d)[D u \cdot (-D d) + u(x)  - \varphi (x) ] - (A + \frac{1} {2}M)h(x),
\end{align}
where
\begin{align*}
\beta =&\max \{\frac{1}{\mu}, \frac{5 n}{2} (2\kappa_{max}+\frac{1}{n}) \frac{C_9}{c_1} \}, \\
A=& \max\{A_1, A_2, A_3, A_4,\frac{C_{20}}{c_4} \}.
\end{align*}

It is easy to know $\widetilde{P} \geq 0$ on $\partial \Omega_\mu$. Precisely, on $\partial \Omega$, we have $d = h= 0$, and $- D d = \nu$, so we can get
\begin{align}\label{5.54}
\widetilde{P}(x) =0,  \quad \text{ on } \partial \Omega.
\end{align}
On $\partial \Omega_\mu \setminus \partial \Omega$, we have $d = \mu$, and
\begin{align}\label{5.55}
\widetilde{P}(x) \geq& - (1+ \beta \mu) [|D u|  +  |u| + |\varphi|]  - (A + \frac{1} {2}M) [-\mu + \mu^2] \notag \\
\geq& - (1+ \beta \mu) [|D u|  +  |u| + |\varphi|]  + \frac{9}{10} \mu A > 0,
\end{align}
since $A \geq \frac{10}{9}(\frac{1}{\mu} + \beta) [|D u|_{C^0}  +  |u|_{C^0} + |\varphi|_{C^0}] +1 =:A_1$. In the following, we want to prove $\widetilde{P}$ attains its minimum only on $\partial \Omega$. Then we can get
\begin{align}\label{5.56}
0 \geq \widetilde{P}_\nu (\widetilde{z_0})  =&  [u_{\nu \nu}(\widetilde{z_0}) - \sum\limits_j {u_j d_{j\nu} }   + u_\nu  - \varphi_\nu] - (A + \frac{1} {2}M)  \notag\\
\geq& \max_{\partial \Omega} u_{\nu \nu} - |D u| |D^2 d| -|D u|  - |D \varphi| - A  - \frac{1} {2}M,
\end{align}
hence \eqref{5.51} holds.

To prove $\widetilde{P}$ attains its minimum only on $\partial \Omega$, we assume $\widetilde{P}$ attains its minimum at some point  $\widetilde{x_0} \in \Omega_\mu$ by contradiction. Rotating the coordinates, we can assume
\begin{align}\label{5.57}
D^2 u(\widetilde{x_0}) \text{ is diagonal}.
\end{align}
In the following, all the calculations are at $\widetilde{x_0}$.

Firstly, we have
\begin{align}\label{5.58}
0 = \widetilde{P}_i =& \beta d_i [-\sum\limits_j {u_j d_j }  + u- \varphi ] + (1+ \beta d)[-\sum\limits_j {(u_{ji} d_j + u_j d_{ji})}   + u_i- \varphi_i ] \notag \\
&- (A + \frac{1} {2}M)h_i\notag \\
 =&\beta d_i [-\sum\limits_j {u_j d_j }  + u- \varphi ] + (1+ \beta d)[-u_{ii} d_i-\sum\limits_j { u_j d_{ji}}   + u_i- \varphi_i ] \notag \\
& - (A + \frac{1} {2}M)h_i,
\end{align}
and
\begin{align}\label{5.59}
0 \leq \widetilde{P}_{ii}  =& \beta d_{ii} [-\sum\limits_j {u_j d_j } + u- \varphi] + 2\beta d_i [-\sum\limits_j {(u_{ji} d_j + u_j d_{ji})}   + u_i- \varphi_i]  \notag \\
&+ (1+ \beta d)[-\sum\limits_j {(u_{jii} d_j +2 u_{ji} d_{ji} + u_j d_{jii})} + u_{ii} - \varphi_{ii} ] - (A + \frac{1} {2}M)h_{ii}   \notag \\
=&  \beta d_{ii} [-\sum\limits_j {u_j d_j } + u- \varphi] + 2\beta d_i [-u_{ii} d_i-\sum\limits_j {u_j d_{ji}}   + u_i- \varphi_i]  \notag \\
&+ (1+ \beta d)[-\sum\limits_j {u_{jii} d_j }  - 2u_{ii} d_{ii}  - \sum\limits_j {u_j d_{jii} }  + u_{ii} - \varphi_{ii}]  \notag \\
&- (A + \frac{1} {2}M)h_{ii}   \notag \\
\leq&  - 2\beta u_{ii} d_i ^2  + (1+ \beta d)[-\sum\limits_j {u_{jii} d_j }  - 2u_{ii} d_{ii} + u_{ii}] \notag \\
&- (A + \frac{1} {2}M)h_{ii} + C_{19},
\end{align}
where $C_{19}$ is a positive constant under control as follows
\begin{align}\label{5.60}
C_{19} =&\beta |D^2 d|_{C^0} \Big[|D u|_{C^0} + |u|_{C^0} + |\varphi|_{C^0}] +2\beta [|D u|_{C^0} |D^2 d|_{C^0}  +|D u|_{C^0}+ |D \varphi|_{C^0}\Big] \notag \\
&+ (1+ \beta \mu)\Big[ |D u|_{C^0} |D^3 d|_{C^0}+ |D^2 \varphi|_{C^0}\Big].
\end{align}

Since $D^2 u(\widetilde{x_0})$ is diagonal, we know $F^{ij} =0$ for $i \ne j$. From the equation \eqref{1.2}, we have
\begin{align}
 \sum\limits_{i = 1}^n {F^{ii} u_{ii} }  =& (k-l) f >0, \notag \\
\sum\limits_{i = 1}^n {F^{ii} u_{iij} }  =&  f_j, \notag
\end{align}
hence
\begin{align}\label{5.61}
0 \leq& \sum\limits_{i = 1}^n {F^{ii} \widetilde{P}_{ii} } \notag \\
\leq&  - 2\beta \sum\limits_{i = 1}^n {F^{ii} u_{ii} d_i ^2 }  + (1+ \beta d)[-\sum\limits_{i,j} {F^{ii} u_{jii} d_j }  - 2\sum\limits_{i = 1}^n {F^{ii} u_{ii} d_{ii} } +\sum\limits_{i = 1}^n {F^{ii} u_{ii} }]  \notag \\
&- (A + \frac{1} {2}M)\sum\limits_{i = 1}^n {F^{ii} h_{ii} }  + C_1\sum\limits_{i = 1}^n {F^{ii} }   \notag \\
\leq&  - 2\beta \sum\limits_{i = 1}^n {F^{ii} u_{ii} d_i ^2 }  - 2(1+ \beta d)\sum\limits_{i = 1}^n {F^{ii} u_{ii} d_{ii} }   \notag \\
&+ [-(A + \frac{1} {2}M)c_4  + C_{20}](\sum\limits_{i = 1}^n {F^{ii} }  + 1),
\end{align}
where $C_{20} = \max \{C_{19}, (1+ \beta \mu)[|D f|_{C^0} +(k-l)|f|_{C^0}]\}$.

Denote $B = \{ i:\beta d_i ^2  < \frac{1}{n},1 \leq i \leq n\}$ and $G = \{ i:\beta d_i ^2  \geq  \frac{1}{n},1 \leq i \leq n\}$. We choose $\beta \geq \frac{1}{\mu} >1$, so
\begin{align}\label{5.62}
d_i ^2 < \frac{1}{n} = \frac{1}{n} |D d|^2, \quad i \in B.
\end{align}
It holds $ \sum_{i \in B} d_i ^2 < 1 = |D d|^2$, and $G$ is not empty. Hence for any $i \in G$, it holds
\begin{align}\label{5.63}
d_{i} ^2 \geq \frac{1}{n\beta}.
\end{align}
and from \eqref{5.58}, we have
\begin{align}\label{5.64}
u_{i i }  =  \frac{{1-2d}}{1+ \beta d}(A + \frac{1}{2}M) + \frac{{\beta [-\sum\limits_j {u_j d_j }  + u- \varphi ]}}
{1+ \beta d} + \frac{{-\sum\limits_j { u_j d_{ji}}   + u_i- \varphi_i }}{{d_{i } }}.
\end{align}
We choose $A \geq 5 \beta \Big[| D u|_{C^0} + |u|_{C^0}+ |\varphi|_{C^0}] + 5 \sqrt{n \beta} [|D u|_{C^0} |D^2 d|_{C^0} + |D u|_{C^0} + |D \varphi|_{C^0}\Big]=:A_2$, such that for any $i \in G$
\begin{align}\label{5.65}
&\Big| \frac{{\beta [-\sum\limits_j {u_j d_j }  + u- \varphi ]}}{1+ \beta d} + \frac{{-\sum\limits_j { u_j d_{ji}}   + u_i- \varphi_i }}{{d_{i } }}\Big|  \notag \\
\leq& \beta \Big[| D u| + |u|+ |\varphi|] + \sqrt{n \beta} [|D u| |D^2 d| + |D u| + |D \varphi|\Big]  \notag \\
\leq& \frac{A}{5},
\end{align}
then we can get
\begin{align}\label{5.66}
\frac{3A}{5}+ \frac{2M}{5} \leqslant u_{i i }  \leq \frac{6A}{5}+ \frac{M}{2}, \quad \forall \quad i \in G.
\end{align}
Also there is an $i_0  \in G$ such that
\begin{align}\label{5.67}
d_{i_0} ^2 \geq \frac{1}{n} |D d|^2 = \frac{1}{n}.
\end{align}
From \eqref{5.61}, we have
\begin{align}\label{5.68}
0 \leq \sum\limits_{i = 1}^n {F^{ii} \widetilde{P}_{ii} }  \leq&  - 2\beta \sum\limits_{i \in G} {F^{ii} u_{ii} d_i ^2 }  - 2\beta \sum\limits_{i \in B} {F^{ii} u_{ii} d_i ^2 }  \notag \\
&- 2(1+ \beta d)\sum\limits_{u_{ii}  > 0} {F^{ii} u_{ii} d_{ii} }  - 2(1+ \beta d)\sum\limits_{u_{ii}  < 0} {F^{ii} u_{ii} d_{ii} }  \notag \\
&+ [-(A + \frac{1} {2}M)c_4  + C_{20}](\sum\limits_{i = 1}^n {F^{ii} }  + 1) \notag \\
\leq&  - 2\beta \sum\limits_{i \in G} {F^{ii} u_{ii} d_i ^2 }  - 2\beta \sum\limits_{i \in B} {F^{ii} u_{ii} d_i ^2 }  \notag \\
&+4 \kappa_{max} \sum\limits_{u_{ii}  > 0} {F^{ii} u_{ii} } \notag \\
&+ [-(A + \frac{1} {2}M)c_4 + C_{20}](\sum\limits_{i = 1}^n {F^{ii} }  + 1),
\end{align}
where $\kappa_{max}$ is defined as in \eqref{5.25}. Direct calculations yield
\begin{align}\label{5.69}
 - 2\beta \sum\limits_{i \in G} {F^{ii} u_{ii} d_i ^2 }  \leq  - 2\beta F^{i_0 i_0 } u_{i_0 i_0 } d_{i_0 } ^2  \leq  - \frac{2\beta }
{{n}}F^{i_0 i_0 } u_{i_0 i_0 },
\end{align}
and
\begin{align}\label{5.70}
- 2\beta \sum\limits_{i \in B} {F^{ii} u_{ii} d_i ^2 }  \leq&  - 2\beta \sum\limits_{i \in B, u_{ii}  < 0} {F^{ii} u_{ii} d_i ^2 }
\leq - \frac{2}{n}\sum\limits_{i \in B,u_{ii} < 0} {F^{ii} u_{ii} }  \notag \\
\leq&  - \frac{2}{n}\sum\limits_{u_{ii} < 0} {F^{ii} u_{ii} }  =- \frac{2}{n}[(k-l)f - \sum\limits_{u_{ii} > 0} {F^{ii} u_{ii} }] \notag \\
\leq& \frac{2}{n} \sum\limits_{u_{ii} > 0} {F^{ii} u_{ii} }.
\end{align}

So it holds
\begin{align}\label{5.71}
0 \leq \sum\limits_{i = 1}^n {F^{ii} \widetilde{P}_{ii} }  \leq& - \frac{2\beta }
{{n}}F^{i_0 i_0 } u_{i_0 i_0 }  +  (4\kappa_{max}+ \frac{2}{n}) \sum\limits_{u_{ii}  > 0} {F^{ii} u_{ii} } \notag \\
&+ [- (A + \frac{1} {2}M)c_4 + C_{20}](\sum\limits_{i = 1}^n {F^{ii} }  + 1).
\end{align}
We divide into three cases to prove the result. Without generality, we assume that $i_0=1 \in G$, and $u_{22} \geq \cdots \geq u_{nn}$.

$\blacklozenge$ CASE I: $u_{nn} \geq 0$.

In this case, we have
\begin{align}\label{5.72}
(4\kappa_{max}+ \frac{2}{n}) \sum\limits_{u_{ii}  > 0} {F^{ii} u_{ii} } =  (4\kappa_{max}+ \frac{2}{n}) \sum\limits_{i =1}^n {F^{ii} u_{ii} }=   (k-l) (4\kappa_{max}+ \frac{2}{n}) f.
\end{align}
Hence from \eqref{5.71} and \eqref{5.72}
\begin{align}\label{5.73}
0 \leq \sum\limits_{i = 1}^n {F^{ii} \widetilde{P}_{ii} }  \leq& (4\kappa_{max}+ \frac{2}{n}) \sum\limits_{u_{ii}  > 0} {F^{ii} u_{ii} }+ [- (A + \frac{1} {2}M)c_4 + C_{20}](\sum\limits_{i = 1}^n {F^{ii} }  + 1) \notag \\
\leq&  (k-l) (4\kappa_{max}+ \frac{2}{n}) f + [- (A + \frac{1} {2}M)c_4 + C_{20}] \notag \\
<& 0,
\end{align}
since $A \geq \frac{  (k-l) (4\kappa_{max}+ \frac{2}{n}) |f|_{C^0} +C_{20}}{c_4} =: A_3$. This is a contradiction.

$\blacklozenge$ CASE II: $u_{nn} < 0$ and $-u_{nn} < \frac{c_4}{10 (4\kappa_{max}+ \frac{2}{n})}u_{11}$.

In this case, we have
\begin{align}\label{5.74}
(4\kappa_{max}+ \frac{2}{n}) \sum\limits_{u_{ii}  > 0} {F^{ii} u_{ii} } =& (8\kappa_{max}+ \frac{2}{n}) [ (k-l)f - \sum\limits_{u_{ii} < 0} {F^{ii} u_{ii} }] \notag  \\
\leq& (4\kappa_{max}+ \frac{2}{n}) [ (k-l)f - u_{nn}\sum\limits_{i=1}^n {F^{ii} }] \notag  \\
<& (4\kappa_{max}+ \frac{2}{n})  (k-l)f +  \frac{c_4}{10}u_{11} \sum\limits_{i=1}^n {F^{ii} } \notag  \\
\leq& (4\kappa_{max}+ \frac{2}{n})  (k-l)f + \frac{c_4}{10} (\frac{6A}{5}+ \frac{M}{2} )\sum\limits_{i=1}^n {F^{ii} }.
\end{align}
Hence from \eqref{5.71} and \eqref{5.74}
\begin{align}\label{5.75}
0 \leq \sum\limits_{i = 1}^n {F^{ii} \widetilde{P}_{ii} }  \leq&  (4\kappa_{max}+ \frac{2}{n}) \sum\limits_{u_{ii}  > 0} {F^{ii} u_{ii} } \notag \\
&+ [- (A + \frac{1} {2}M)c_4 + C_{20}](\sum\limits_{i = 1}^n {F^{ii} }  + 1) \notag \\
<&  (4\kappa_{max}+ \frac{2}{n})  (k-l)f + \frac{c_4}{10} (\frac{6A}{5}+ \frac{M}{2} )\sum\limits_{i=1}^n {F^{ii} } \notag \\
&+ [- (A + \frac{1} {2}M)c_4 + C_{20}](\sum\limits_{i = 1}^n {F^{ii} }  + 1) \notag \\
<&0,
\end{align}
since $A \geq \max\{\frac{  (k-l) (4\kappa_{max}+ \frac{2}{n}) |f|_{C^0} +C_{20}}{c_4}, \frac{25 C_{20}}{3c_4} \} =: A_4$. This is a contradiction.

$\blacklozenge$ CASE III: $u_{nn} < 0$ and $-u_{nn} \geq \frac{c_4}{10 (4\kappa_{max}+ \frac{2}{n})} u_{11} $.

In this case, we have $u_{11} \geq \frac{3A}{5}+ \frac{2M}{5}$, and $u_{22} \leq C_9 (1+M)$. So
\begin{align} \label{5.76}
u_{11} \geq \frac{2}{5C_9} u_{22}.
\end{align}
Let $\delta = \frac{2}{5C_9}$ and $\varepsilon =\frac{c_4}{10 (4\kappa_{max}+ \frac{2}{n})}$, \eqref{2.15} in Lemma \ref{lem2.7} holds, that is
\begin{align} \label{5.77}
F^{11} \geq c_1 \sum\limits_{i = 1}^n {F^{ii} },
\end{align}
where $c_1 = \frac{n} {k}\frac{{k - l}}{{n - l}}\frac{c_0^2} {{n - k + 1}}$ and  $c_0 = \min \{ \frac{{\varepsilon ^2 \delta ^2}}{{2(n - 2)(n - 1)}},\frac{{\varepsilon ^2 \delta }}{{4(n- 1)}}\}$. Hence from \eqref{5.71} and \eqref{5.77}
\begin{align}\label{5.78}
0 \leq \sum\limits_{i = 1}^n {F^{ii} \widetilde{P}_{ii} }  \leq& - \frac{2\beta }
{{n}}F^{11 } u_{11}  +  (4\kappa_{max}+ \frac{2}{n}) \sum\limits_{u_{ii}  > 0} {F^{ii} u_{ii} } \notag \\
&+ [- (A + \frac{1} {2}M)c_4 + C_{20}](\sum\limits_{i = 1}^n {F^{ii} }  + 1)\notag \\
\leq& - \frac{2\beta } {{n}}c_1  (\frac{3A}{5}+ \frac{2M}{5}) \sum\limits_{i = 1}^n {F^{ii} } + (4\kappa_{max}+ \frac{2}{n}) C_9(1 + M) \sum\limits_{i = 1}^n {F^{ii} }\notag \\
<& 0,
\end{align}
since $\beta \geq \frac{5n}{2} (2\kappa_{max}+\frac{1}{n}) \frac{C_9}{c_1}$. This is a contradiction.

So $\widetilde{P}$ attains its maximum only on $\partial \Omega$. The proof of Lemma \ref{lem5.5} is complete.

\end{proof}

Following above proofs, we can also obtain the estimates of second order derivatives of $u^\varepsilon$ in \eqref{1.4}, and the strict convexity of $\Omega$ is important in reducing global second derivatives to double normal second derivatives on boundary. So we have

\begin{theorem} \label{th5.6}
Suppose $\Omega \subset \mathbb{R}^n$ is a $C^4$ strictly convex domain, $f \in C^2(\overline{\Omega})$ is a positive function, $\varphi \in C^3(\partial \Omega)$ and $u^\varepsilon \in C^4(\Omega)\cap C^3(\overline \Omega)$ is the $k$-admissible solution of Hessian quotient equation \eqref{1.4} with $\varepsilon >0$ sufficiently small,, then we have
\begin{align}\label{5.79}
\sup_{\Omega} |D^2 u^\varepsilon |  \leq \overline{M_2},
\end{align}
where $\overline{M_2}$ depends on $n$, $k$, $l$, $\Omega$, $\overline{M_1}$, $\inf f$, $|f|_{C^2}$ and $|\varphi|_{C^3}$.
\end{theorem}

\section{Existence of the boundary problems}

In this section we complete the proof of the Theorem \ref{th1.1} and Theorem \ref{th1.2}.

\subsection{Proof of Theorem \ref{th1.1}}

For the Neumann problem of Hessian quotient equations \eqref{1.2}, we have established the $C^0$, $C^1$ and $C^2$ estimates in Section 3, Section 4, Section 5, respectively. By the global $C^2$ a priori estimates, Hessian quotient equations \eqref{1.2} is uniformly elliptic in $\overline \Omega$.
Due to the concavity of Hessian quotient operator $[\frac{\sigma _k (\lambda)}{\sigma _l (\lambda)}]^{\frac{1}{k-l}}$ in $\Gamma_k$, we can get the global H\"{o}lder estimates of second derivative following the discussions in \cite{LT86}, that is, we can get
\begin{align}\label{6.1}
|u|_{C^{2,\alpha}(\overline \Omega)} \leq C,
\end{align}
where $C$ and $\alpha$ depend on $n$, $k$, $l$, $\Omega$, $\inf f$, $|f|_{C^2}$ and $|\varphi|_{C^3}$.
From \eqref{6.1} one also obtains $C^{3,\alpha}(\overline \Omega)$ estimates by differentiating the equation
\eqref{1.2} and apply the Schauder theory for linear, uniformly elliptic equations.

Applying the method of continuity (see \cite{GT}, Theorem 17.28), the existence of the classical solution holds.
By the standard regularity theory of uniformly elliptic partial differential equations, we can obtain the higher regularity.

\subsection{Proof of Theorem \ref{th1.2}}

The proof is following the idea of Qiu-Xia \cite{QX16}.

By a similar proof of Theorem \ref{th1.1}, we know there exists a unique $k$-admissible solution $u^\varepsilon \in C^{3, \alpha}(\overline \Omega)$ to \eqref{1.4} for any small $\varepsilon >0$. Let $v^\varepsilon =u^\varepsilon - \frac{1}{|\Omega|} \int_\Omega u^\varepsilon$, and it is easy to know $v^\varepsilon$ satisfies
\begin{align}
\left\{ \begin{array}{l}
\frac{\sigma _k (D^2 v^\varepsilon)}{\sigma _l (D^2 v^\varepsilon)} = f(x),  \quad \text{in} \quad \Omega,\\
(v^\varepsilon)_\nu = - \varepsilon v^\varepsilon - \frac{1}{|\Omega|} \int_\Omega \varepsilon u^\varepsilon + \varphi(x),\qquad \text{on} \quad \partial \Omega.
 \end{array} \right.
\end{align}
By the global gradient estimate \eqref{4.36}, it is easy to know $\varepsilon \sup |D u^\varepsilon | \rightarrow 0$. Hence there is a constant $c$ and a function $v \in C^{2}(\overline \Omega)$, such that $-\varepsilon u^\varepsilon \rightarrow c$, $-\varepsilon v^\varepsilon \rightarrow 0$, $\frac{1}{|\Omega|} \int_\Omega \varepsilon u^\varepsilon \rightarrow c$ and $v^\varepsilon \rightarrow v$ uniformly in $C^{2}(\overline \Omega)$ as $\varepsilon \rightarrow 0$. It is easy to verify that $v$ is a solution of
\begin{align}
\left\{ \begin{array}{l}
\frac{\sigma _k (D^2 v)}{\sigma _l (D^2 v)} = f(x),  \quad \text{in} \quad \Omega,\\
v_\nu = c + \varphi(x),\qquad \text{on} \quad \partial \Omega.
 \end{array} \right.
\end{align}
If there is another function $v_1 \in C^{2}(\overline \Omega)$ and another constant $c_1$ such that
\begin{align}
\left\{ \begin{array}{l}
\frac{\sigma _k (D^2 v_1)}{\sigma _l (D^2 v_1)} = f(x),  \quad \text{in} \quad \Omega,\\
(v_1)_\nu = c_1 + \varphi(x),\qquad \text{on} \quad \partial \Omega.
 \end{array} \right.
\end{align}
Applying the maximum principle and Hopf Lemma, we can know $c = c_1$ and $v - v_1$ is a constant.
By the standard regularity theory of uniformly elliptic partial differential equations, we can obtain the higher regularity.

\textbf{Acknowledgement}
 The authors would like to express sincere gratitude to Prof. Xi-Nan Ma for the constant encouragement in this subject.

\end{document}